\documentclass[a4paper]{article}
\usepackage{amsmath,amssymb,latexsym,theorem,graphicx,mathrsfs}
\usepackage{secdot}

\usepackage{xcolor}

\usepackage[colorlinks]{hyperref}
\definecolor{ForestGreen}{rgb}{0.1,0.6,0.05}
\definecolor{EgyptBlue}{rgb}{0.063,0.1,0.6}
\hypersetup{
	colorlinks=true,
	linkcolor=EgyptBlue,         
    citecolor=ForestGreen,
    urlcolor={green}
}

\usepackage{epstopdf}
\newtheorem{theorem}{Theorem}
\newtheorem{proposition}[theorem]{Proposition}
\newtheorem{lemma}[theorem]{Lemma}

\newtheorem{remark}[theorem]{Remark}
\newtheorem{definition}[theorem]{Definition}

\numberwithin{equation}{section}
\numberwithin{theorem}{section}

\newcommand{\qed}{{\unskip\nobreak\hfil%
                 \penalty50\hskip .001pt\hbox{}\nobreak\hfil
                 \vrule height 1.7ex width .9ex depth .2ex
                 \parfillskip=0pt\finalhyphendemerits=0\medbreak}\rm}
\newenvironment{proof}{\begin{trivlist}\item[\hskip%
                      \labelsep{{\em Proof.}\ }]\rm}%
                      {\hfill\qed\rm\end{trivlist}}

\newenvironment{proof*}[1]{\begin{trivlist}\item[\hskip%
                          \labelsep{{\bf Proof of \/{\rm\bf #1.}}\quad}]\rm}%
                          {\hfill\qed\rm\end{trivlist}}

\newcommand{\W}{W_0^{1,p}}
\newcommand{\intO}{\int_\Omega}
\newcommand{\C}{C^1_0(\overline{\Omega})}
\newcommand{\I}{I_{\alpha,\beta}}
\newcommand{\f}{f_{\alpha,\beta}}

\newcommand{\E}{E_{\alpha,\beta}}

\title{On positive solutions for $(p,q)$-Laplace equations 
with two parameters
\footnote{AMS Subject Classifications: 35J62, 35J20, 35P30}}
\author{ 
\normalsize Vladimir Bobkov\footnote{Corresponding author}
\footnote{The first author was partially supported by the Russian Foundation for Basic Research (projects No. 13-01-00294 and 14-01-31054)}\\ 
{\small  Institute of Mathematics, Ufa Science Center of RAS}\\
{\small Chernyshevsky str. 112, Ufa 450008, Russia}\\
{\small e-mail: bobkovve@gmail.com}\\[0.5em] 
\normalsize Mieko Tanaka\\
{\small Department of  Mathematics, 
Tokyo University of Science}\\
{\small Kagurazaka 1-3, Shinjyuku-ku, Tokyo 162-8601, Japan}\\
{\small e-mail: tanaka@ma.kagu.tus.ac.jp}
}

\date{}

\begin{document}
\maketitle 

\begin{abstract}
We study the existence and non-existence of positive solutions for the $(p,q)$-Laplace equation
$-\Delta_p u -\Delta_q u=\alpha |u|^{p-2}u+\beta |u|^{q-2}u$, where $p \neq q$, 
under the zero Dirichlet boundary condition in $\Omega$. The main result
of our research is the construction of a continuous curve in $(\alpha,\beta)$ plane, 
which becomes a threshold between the existence and non-existence of positive solutions.
Furthermore, we provide the example of domains $\Omega$
for which the corresponding first Dirichlet eigenvalue of $-\Delta_p$ is not monotone w.r.t. $p > 1$.

\par
\smallskip
\noindent {\bf  keywords}:\  $(p,q)$-Laplacian,\ nonlinear
eigenvalue problems, global minimizer,\ mountain pass theorem,\ 
Nehari manifold, \ super- and sub-solution, \ modified Picone's  identity, \ extended functional
\end{abstract}

\section{Introduction}
In this article 
we are concerned with
the existence and non-existence of positive solutions for 
the following $(p,q)$-Laplace equation: 
\begin{equation*}
\left\{
\begin{aligned}
-&\Delta_p u -\Delta_q u=\alpha |u|^{p-2}u+\beta |u|^{q-2}u 
&&{\rm in}\ \Omega, \\
&u=0 &&{\rm on }\ \partial \Omega,
\end{aligned}
\right.
\leqno{(GEV;\alpha,\beta)}
\end{equation*}
where $\Omega\subset \mathbb{R}^N$ ($N \geq 1$) is a bounded domain with $C^2$-boundary $\partial \Omega$, $\alpha, \beta \in\mathbb{R}$ and $1<q<p<\infty$. Note that the assumption $q<p$ is taken without loss of generality, due to the symmetry of symbols in $(GEV;\alpha,\beta)$; therefore all results of the present work have corresponding counterparts in the case $q > p$.

Operator $\Delta_r$ stands for the usual
$r$-Laplacian, i.e., $\Delta_r u :={\rm div}\,(|\nabla u|^{r-2}\nabla
u)$ with $r\in(1,+\infty)$. Hereinafter, by $W_0^{1,r} := W_0^{1,r}(\Omega)$ we denote the standard Sobolev space, and $\| \cdot \|_r$ denotes the norm in $L^r(\Omega)$.

We say that $u\in\W$ is a (weak) solution of $(GEV;\alpha,\beta)$ if it holds
\begin{equation*}
\intO |\nabla u|^{p-2}\nabla u\nabla\varphi \,dx 
+\intO |\nabla u|^{q-2}\nabla u\nabla\varphi \,dx 
= \intO (\alpha 
|u|^{p-2}u+\beta |u|^{q-2}u) \, \varphi\,dx
\end{equation*}
for all $\varphi\in \W$. 

\smallskip
In the last few years, the $(p,q)$-Laplace operator attracts a lot of attention and has been studied by many authors (cf. \cite{CherIl, GoGu, Ramos, WeYand2009}). 
However, there are only few results regarding the eigenvalue problems for the 
$(p,q)$-Laplacian. 
The study of the problem $(GEV;\alpha,\beta)$ started in the form of a perturbation of 
{\it homogeneous} eigenvalue problem 
(see \cite{BB,BB-2,Mih,T-2014}). 
Recently, Motreanu and the second author in \cite{MT} introduced the eigenvalue problem 
\begin{equation}
\label{eq:MT}
\left\{
\begin{aligned}
-&\Delta_p u -\Delta_q u=\lambda (m_p(x)|u|^{p-2}u+m_q(x)|u|^{q-2}u)
&& {\rm in}\ \Omega, \\
&u=0 && {\rm on }\ \partial \Omega,
\end{aligned}\right.
\end{equation}
where indefinite weights $m_p, m_q\in L^\infty(\Omega)$ are
such that the Lebesgue measure of $\{x\in\Omega :~m_r(x)>0\}$ is positive 
($r=p, q$). 

In \cite{KTT}, by using the time map, 
Kajikiya et al. provided five typical examples 
of the bifurcation 
of positive solutions for the one-dimensional $(p,q)$-Laplace 
equation on the interval $(-L,L)$: 
\begin{equation*}
\left\{
\begin{aligned}
&(|u'|^{p-2}u')' + (|u'|^{q-2}u')' + \lambda(|u|^{p-2}u+|u|^{q-2}u)=0, \\
&u(-L)=u(L)=0. 
\end{aligned}
\right.
\end{equation*}
They have shown that the bifurcation curve changes depending on 
$p$, $q$ and $L$. 

Investigation of the problem $(GEV;\alpha,\beta)$ with \textit{two} spectral parameters, on the one hand, generalizes and complements the research \cite{MT}, and seems more natural, due to the structure of the equation.
 We restrict ourselves to the case where $m_p$ and $m_q$ are constants, to save transparency and simplicity of presentation.
 However, we emphasize that all the results of the present article remain valid for the problem
$(GEV;\alpha,\beta)$ with \textit{non-negative} weights:
$$
\left\{
\begin{aligned}
-&\Delta_p u -\Delta_q u=\alpha m_p(x) |u|^{p-2}u + \beta m_q(x) |u|^{q-2}u 
&& {\rm in}\ \Omega, \\
&u=0 && {\rm on }\ \partial \Omega,
\end{aligned}
\right.
$$
where $m_r \in L^\infty(x)$, $m_r \not\equiv 0$ and $m_r \geq 0$ a.e. in $\Omega$ for $r = p, q$. 

On the other hand, the statement of the problem $(GEV;\alpha,\beta)$ has been inspired by the conception of the Fu\v{c}ik spectrum, which, in terms of the $p$-Laplacian, consists of finding all spectral points $(\alpha, \beta)$ such that the problem
$$
\left\{
\begin{aligned}
-&\Delta_p u = \alpha u_+^{p-1} - \beta u_-^{p-1}
&& {\rm in}\ \Omega, \\
&u = 0 && {\rm on }\ \partial \Omega
\end{aligned}
\right.
$$
possesses non-trivial solutions (cf. \cite{cuestafucik}). 
The set of all such points $(\alpha,\beta)$ is called the Fu\v{c}ik spectrum for the $p$-Laplacian. 
Here we denote $u_\pm := \max\{\pm u, 0\}$ in $\Omega$. 
From a physical point of view, the values $\alpha$ and $\beta$ can be seen as a contribution of $u_+$ and $u_-$ to the steady-state behavior of the corresponding nonlinear reaction-diffusion equation.

\smallskip
In the present article we provide a \textit{complete} description of  $2$-dimensional sets in the $(\alpha, \beta)$ plane, which correspond
 to the existence and non-existence of positive solutions for $(GEV;\alpha,\beta)$, see Fig.~1 and Section 2 for precise statements. Moreover, we also give a description of the principal $1$-dimensional sets. The main result here
  is the construction of a continuous threshold curve $\mathcal{C}$, which separates the regions of the existence and non-existence of positive solutions for $(GEV;\alpha,\beta)$.

\begin{figure}[h!]
\center{\includegraphics[scale=1]{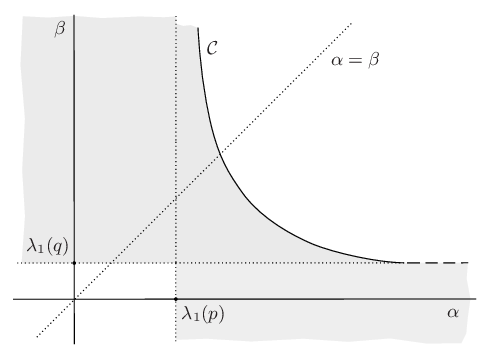} \\ a) {\bf (LI)} holds}
\vfill
\center{\includegraphics[scale=1]{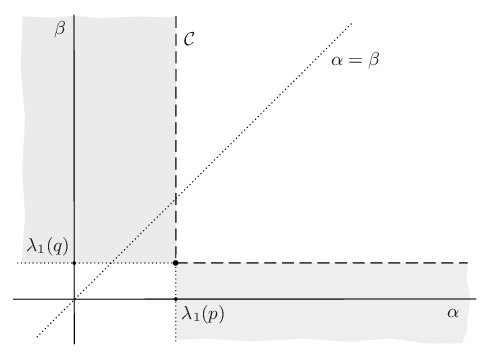} \\ b) {\bf (LI)} is violated}
\caption{Shaded sets correspond to existence, unshaded to non-existence}
\end{figure}


\smallskip
As usual, we say that $\lambda$ is an eigenvalue of $-\Delta_r$ with
weight function $m_r\in L^\infty(\Omega)$ 
if the problem
\begin{equation*}
\left\{
\begin{aligned}
-&\Delta_r u =\lambda m_r(x)|u|^{r-2}u
&& {\rm in}\ \Omega, \\
&u=0 && {\rm on }\ \partial \Omega.
\end{aligned}
\right.
\leqno{(EV;r,\lambda)}
\end{equation*}
has a non-trivial solution. If the Lebesgue measure of $\{x\in\Omega :~m_r(x)>0\}$ is positive, then $(EV;r,\lambda)$ possesses the first positive eigenvalue
$\lambda_1(r,m_r)$ (cf. \cite{anane1987}), that can be
obtained by minimizing the Rayleigh quotient:
\begin{equation*}
\lambda_1(r,m_r)
:=\inf\left\{\frac{\intO |\nabla u|^r\,dx}{\intO m_r|u|^r\,dx}
\,:~u\in W_0^{1,r},\ \intO m_r|u|^r\,dx>0\right\}.
\end{equation*}
We note that $\lambda_1(r,m_r)$ is simple and isolated, and the corresponding eigenfunction $\varphi_1(r,m_r)$ belongs to $C^{1,\alpha_r}_0(\overline{\Omega})$, where $\alpha_r\in(0,1)$.
Hereinafter we will also use the notation $\lambda_1(r)$ for the first eigenvalue of $-\Delta_r$ without weight (that is, with $m_r\equiv 1$) and $\varphi_r$ for the corresponding eigenfunction.

\smallskip
In what follows, we will say that $\lambda_1(p)$ and $\lambda_1(q)$ 
have different eigenspaces if the
corresponding eigenfunctions $\varphi_p$ and $\varphi_q$ are linearly independent, i.e. the following assumption is satisfied:
\begin{itemize}
\item[{\bf (LI)}]
For any $k \neq 0$ it holds $\varphi_p \not\equiv k \varphi_q$ in $\Omega$.
\end{itemize}

Let us note that availability or violation of the assumption {\bf (LI)} significantly affects the sets of existence of solutions for $(GEV;\alpha,\beta)$, see Fig.~1 and the next section for precise statements. 

Although it is not shown, to the best of our knowledge, that {\bf (LI)} generally holds,
in the one-dimensional case $\lambda_1(p)$ and $\lambda_1(q)$ have different eigenspaces for $p \neq q$ (cf. [14]).
Therefore, we conjecture that {\bf (LI)} is always satisfied for the eigenvalue problems \textit{without} weights in the general $n$-dimensional case.

At the same time, {\bf (LI)} can be violated, if we consider eigenvalues $\lambda_1(p, m_p)$ and $\lambda_1(q, m_q)$ \textit{with} (non-negative) weights. The corresponding example is given in Appendix C. 
Hence, the breach of the assumption {\bf (LI)} 
may actually occur for the problem $(GEV;\alpha,\beta)$ with non-negative weights, for which, as noted above, all the results of the article hold.

\smallskip
Considering $\lambda_1(r)$ (or $\lambda_1(r, m_r)$) as a function of $r$ there also arise   questions about the behaviour of this function and its geometrical properties. 

It is known from Kajikiya et al. \cite{KTT} 
that in a one-dimensional case, i.e. $\Omega \subset \mathbb{R}$, the first eigenvalue $\lambda_1(r)$ is non-monotone 
w.r.t. $r > 1$ provided $\Omega=(-L,L)$ with $L>1$, that is, there exists a unique 
maximum point $r^*(L)>1$ of $\lambda_1(r)$ 
such that $\lambda_1(r)$ is strictly increasing on 
$(1,r^*(L))$ and strictly decreasing on $(r^*(L),\infty)$. 

In Appendix B we show that the same non-monotonicity of $\lambda_1(r)$ w.r.t. $r > 1$ occurs also in a higher dimensional case.

\smallskip
Now we give three results from \cite{MT}, where they were proved using the variational methods. These results will be used below.

\begin{theorem}[\protect{\cite[Theorem 1]{MT}}]\label{thm:MT-1} 
If it holds
$$
-\min\{\lambda_1(p,-m_p),\lambda_1(q,-m_q)\}< \lambda<
\min\{\lambda_1(p,m_p),\lambda_1(q,m_q)\},
$$
then \eqref{eq:MT} has no non-trivial solutions. 
\end{theorem}

\begin{theorem}[\protect{\cite[Theorem 3]{MT}}]\label{thm:MT-2}
Assume that there holds $\lambda_1(p,m_p)\not=\lambda_1(q,m_q)$
(resp. $\lambda_1(p,-m_p)\not=\lambda_1(q,-m_q)$). If
\begin{gather*}
\min\{\lambda_1(p,m_p),\lambda_1(q,m_q)\}<\lambda
<\max\{\lambda_1(p,m_p),\lambda_1(q,m_q)\}
\\
({\rm resp.}\quad
-\max\{\lambda_1(p,-m_p),\lambda_1(q,-m_q)\}<\lambda
<-\min\{\lambda_1(p,-m_p),\lambda_1(q,-m_q)\}),
\end{gather*}
then \eqref{eq:MT} has at least one positive solution. 
\end{theorem}

\begin{theorem}[\protect{\cite[Theorems 4, 5]{MT}}]\label{thm:MT-3}
Assume that for $(r,r')=(p,q)$ or $(q,p)$, 
\begin{equation}
\label{hyp:MT-3}
\pm \lambda=\lambda_1(r,\pm m_r)>\lambda_1(r',\pm m_{r'})
\end{equation}
and
\begin{equation}
\label{hypo2:MT-3}
\intO |\nabla \varphi_1(r,\pm m_r)|^{r'}\,dx
-\lambda\intO m_{r'} \varphi_1(r,\pm m_r)^{r'}\,dx >0,
\end{equation}
respectively. Then \eqref{eq:MT} has at least one positive
solution.
\end{theorem}

Notice that we set $\lambda_1(r, -m_r) = +\infty$, provided $m_r \geq 0$ in $\Omega$.

Let us mention that in \cite{MT}, there is {\it no} information about the case when 
$\lambda$ is beyond the first eigenvalues $\lambda_1(p,m_p)$ and 
$\lambda_1(q,m_q)$, i.e. $\lambda > \max\{\lambda_1(p,m_p), \lambda_1(q,m_q)\}$.
In the next section we provide corresponding results on existence and non-existence in this case.

\section{Main results} 

First, we state our results for the case $(\alpha, \beta) \in \mathbb{R}^2 \setminus [\lambda_1(p), +\infty) \times [\lambda_1(q), +\infty)$  (see Fig.~1).
These results generalize Theorems 1 and 2 from \cite{MT} for the problem $(GEV;\alpha,\beta)$ with non-negative weights.

\begin{proposition}\label{prop:non-exist} 
If it holds 
$$
(\alpha,\beta)\in (-\infty,\lambda_1(p)]\times(-\infty,\lambda_1(q)]\setminus 
\{(\lambda_1(p),\lambda_1(q))\}, 
$$
then $(GEV;\alpha,\beta)$ has no non-trivial solutions. 

Moreover, 
$(GEV;\alpha,\beta)$ with $\alpha=\lambda_1(p)$ and $\beta=\lambda_1(q)$ 
has a non-trivial (positive) solution 
if and only if they have the same eigenspace, namely, 
there exists $k \neq 0$ such that $\varphi_p \equiv k \varphi_q$ in $\Omega$ (that is, {\bf (LI)} is not satisfied).
\end{proposition}

\begin{proposition}\label{prop:exist-1} 
If it holds 
$$
(\alpha,\beta)\in (\lambda_1(p),+\infty)\times(-\infty,\lambda_1(q)) 
\cup (-\infty,\lambda_1(p))\times(\lambda_1(q),+\infty), 
$$
then $(GEV;\alpha,\beta)$ has at least one positive solution. 
\end{proposition}

\smallskip
The main novelty of the work is the treatment of the rest part of $(\alpha, \beta)$ plane, i.e.
$(\alpha, \beta) \in [\lambda_1(p), +\infty) \times [\lambda_1(q), +\infty)$, where
we construct a threshold curve, which separates the regions of existence and non-existence of positive solutions for $(GEV; \alpha,\beta)$.

Note first that for any $\alpha, \beta \in \mathbb{R}$
the problem $(GEV; \alpha,\beta)$ is equivalent to $(GEV; \beta + s,\beta)$, where $s = \alpha-\beta$. Denoting now, for convenience, $\lambda = \beta$, for each $s \in \mathbb{R}$ we consider
\begin{equation}
\label{def:lambda_1}
\lambda^*(s) :=\sup\{\lambda\in\mathbb{R}\, :\, 
(GEV;\lambda+s,\lambda)\ {\rm has\ a\ positive\ solution}\,\},
\end{equation} 
provided $(GEV;\lambda+s,\lambda)$ 
has a positive solution for some $\lambda$. 
If there are no such $\lambda$, we set 
$\lambda^*(s)=-\infty$. Define also
\begin{equation*}
s^*:= \lambda_1(p)-\lambda_1(q) \quad \text{and} \quad 
s^*_+ := \frac{\|\nabla \varphi_q\|^p_p}{\|\varphi_q\|^p_p} - \lambda_1(q).
\end{equation*} 
Obviously, $s^* \leq s^*_+$, and $s^* = s^*_+$ if and only if {\bf (LI)} is satisfied.

\begin{figure}[h!]
\center{\includegraphics[scale=1]{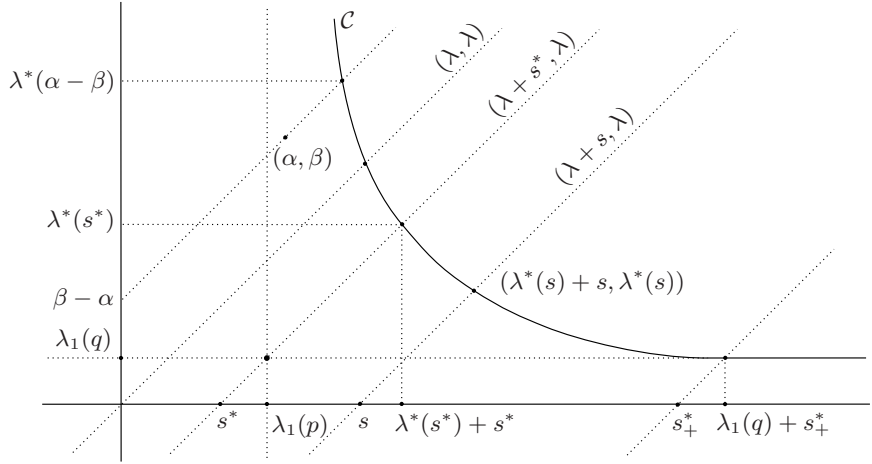}} \\
\caption{$(\alpha, \beta)$ plane. Construction of the curve $\mathcal{C}$}
\end{figure}

In the next proposition we collect the main facts about $\lambda^*(s)$:
\begin{proposition}
\label{prop:prop}
Let $\lambda^*(s)$ be defined by \eqref{def:lambda_1} for $s \in \mathbb{R}$. Then the following assertions hold:

\begin{itemize} 
\item[{\rm (i)}] 
$\lambda^*(s) < +\infty$ for all $s \in \mathbb{R}$;
\item[{\rm (ii)}] 
$\lambda^*(s) + s \geq \lambda_1(p)$ and $\lambda^*(s) \geq \lambda_1(q)$ for all $s \in \mathbb{R}$;
\item[{\rm (iii)}] 
$\lambda^*(s) = \lambda_1(q)$ for all $s \geq s^*_+$;
\item[{\rm (iv)}] 
$\lambda^*(s^*) + s^* > \lambda_1(p)$ and $\lambda^*(s^*) > \lambda_1(q)$ if and only if {\bf (LI)} is satisfied;
\item[{\rm (v)}] 
$\lambda^*(s)$ is continuous on $\mathbb{R}$;
\item[{\rm (vi)}] 
$\lambda^*(s)$ is non-increasing and $\lambda^*(s) + s$ is non-decreasing on $\mathbb{R}$.
\end{itemize} 
\end{proposition}

Notice that it is still unknown if there is $s^*_- \in \mathbb{R}$, such that
$\lambda^*(s) + s = \lambda_1(p)$ for all $s \leq s^*_-$, or $\lambda^*(s) + s > \lambda_1(p)$ for all $s \in \mathbb{R}$, whenever {\bf (LI)} is satisfied.

\smallskip
Now we define the curve $\mathcal{C}$ in $(\alpha, \beta)$ plane as follows: 
$$
\mathcal{C}
:=\left\{(\lambda^*(s)+s,\lambda^*(s)), ~s \in \mathbb{R}\right\}.
$$
From Proposition \ref{prop:prop} there directly follow the corresponding conclusions for $\mathcal{C}$, namely, $\mathcal{C}$ is locally finite, $\mathcal{C} \subset [\lambda_1(p),+\infty)\times[\lambda_1(q),+\infty)$, $\mathcal{C}$ is continuous, monotone, and coincides with $[\lambda_1(q)+s^*_+, +\infty) \times \{\lambda_1(q)\}$ for $s \geq s^*_+$ (see Fig.~2).

We especially note that $\lambda^*(s) + s = \lambda_1(p)$ for $s \leq s^*$ 
and $\lambda^*(s) = \lambda_1(q)$ for $s \geq s^*$ if and only if {\bf (LI)} doesn't hold. It directly follows from the combination of the criterion (iv), estimations (ii) and monotonicity (vi) from Proposition \ref{prop:prop}. In other words, our curve $\mathcal{C}$ coincides with the polygonal line 
$\{\lambda_1(p)\}\times[\lambda_1(q),\infty)\cup [\lambda_1(q),\infty)\times\{\lambda_1(p)\}$ if and only if {\bf (LI)} is violated.

This fact allows us to prove a \textit{complete} description of the spectrum for the problem $(GEV;\alpha,\beta)$ when $\lambda_1(p)$ and $\lambda_1(q)$ 
have the same eigenspace (see Fig.~1 b)).

\begin{theorem}
\label{thm:main2} 
Assume that {\bf (LI)} does not hold. Then $(GEV;\alpha,\beta)$ has at least one positive solution if and only if 
\begin{equation}
\label{eq:condition}
(\alpha,\beta)\in (\lambda_1(p),+\infty)\times(-\infty,\lambda_1(q)) 
\cup (-\infty,\lambda_1(p))\times(\lambda_1(q),+\infty) 
\cup \{(\lambda_1(p),\lambda_1(q))\}.
\end{equation}
\end{theorem}

\smallskip
The second main result is related to the case when $\lambda_1(p)$ and $\lambda_1(q)$ 
have different eigenspaces, and justifies that $\mathcal{C}$ is indeed a \textit{threshold curve} which separates the regions of existence and non-existence of positive solutions for $(GEV; \alpha, \beta)$.

\begin{theorem}\label{thm:main} 
Assume that {\bf (LI)} is satisfied.
If one of the following cases holds (see Fig.~2), then 
$(GEV;\alpha,\beta)$ has at least one positive solution: 
\begin{itemize} 
\item[{\rm (i)}] 
$\lambda_1(q)<\beta$ and $\beta < \lambda^*(\alpha-\beta)$; 
\item[{\rm (ii)}] 
$\lambda_1(p)<\alpha$ and $\beta < \lambda^*(\alpha-\beta)$. 
\end{itemize} 

Conversely, 
if $\beta>\lambda^*(\alpha-\beta)$, then 
$(GEV;\alpha,\beta)$ has no positive solutions. 
\end{theorem} 
This theorem means that $(GEV;\alpha,\beta)$ possesses a positive solution 
if $(\alpha,\beta)$ is below the curve $\mathcal{C}$, and has no positive solutions if $(\alpha,\beta)$ is above $\mathcal{C}$.

We emphasize that Theorem \ref{thm:main} in combination with Propositions \ref{prop:non-exist} and \ref{prop:exist-1} provides a full description of $2$-dimensional sets of existence and non-existence for positive solutions to $(GEV;\alpha, \beta)$ when {\bf (LI)} holds.

\smallskip
Finally, we provide the results about existence and non-existence on the curve $\mathcal{C}$.

\begin{proposition}\label{prop:special-2} 
The following assertions hold:
\begin{itemize} 
\item[{\rm (i)}] 
If $\lambda^*(s)+s >\lambda_1(p)$ and $\lambda^*(s)>\lambda_1(q)$, 
then $(GEV;\lambda^*(s)+s, \lambda^*(s))$ has at least one positive solution; 
\item[{\rm (ii)}] 
If $s > s^*_+$, then $(GEV;\lambda^*(s)+s, \lambda^*(s)) \equiv (GEV;\lambda_1(q)+s, \lambda_1(q))$ has no positive solutions.
\end{itemize} 
\end{proposition}

We remark that the existence of positive solutions remains open in the borderline case $(GEV; \lambda_1(p), \lambda^*(s))$ if the curve $\mathcal{C}$ touches the line $\{\lambda_1(p)\} \times (\lambda_1(q), +\infty)$.

\smallskip
Let us note that the main disadvantage of characterization \eqref{def:lambda_1} of $\lambda^*(s)$ is its non-constructive form. 
However, using the extended functional method (see \cite{ilfunc, BobkovIlyasov2014}) we provide the equivalent characterization of $\lambda^*(s)$ by an explicit minimax formula, which can be used in further numerical investigations of $(GEV; \alpha, \beta)$:
\begin{equation}
\label{def:lambda_2}
\Lambda^*(s) = \sup_{u \in {\rm int}\,\C_+} \inf_{\varphi \in C^1_0(\overline{\Omega})_+ \setminus\{0\}}  \mathcal{L}_s(u; \varphi),
\end{equation}
where 
$$
\mathcal{L}_s(u; \varphi) = \frac{\intO |\nabla u|^{p-2}\nabla u\nabla\varphi \,dx 
+\intO |\nabla u|^{q-2}\nabla u\nabla\varphi \,dx - s \intO 
|u|^{p-2} u \varphi\,dx}{\intO 
|u|^{p-2} u \varphi\,dx + \intO |u|^{q-2} u \varphi\,dx}
$$
and  ${\rm int}\,\C_+$ denotes the interior of the positive cone 
$$
C^1_0(\overline{\Omega})_+= \left\{ u \in C^1_0(\overline{\Omega})\,:\, 
u(x) \geq 0\ \text{ for every } x \in \Omega \, \right\} 
$$
of the Banach space $C^1_0(\overline{\Omega})$, that is, 
\begin{equation}
\label{def:int}
{\rm int}\,\C_+=\left\{u \in C^1_0(\overline{\Omega})\,:\, 
u(x)>0 \text{ for all } x \in \Omega, \,
\frac{\partial u}{\partial\nu}(x) < 0 \text{ for all } x \in \partial\Omega \right\}, 
\end{equation}
with an outer normal $\nu$ to $\partial \Omega$. 

\begin{proposition}
\label{prop:equiv}
$\Lambda^*(s) = \lambda^*(s)$ for all $s \in \mathbb{R}$.
\end{proposition}

\smallskip
Notice that we don't provide information about uniqueness, multiplicity, or existence and non-existence of sign-changing (nodal) solutions for $(GEV; \alpha, \beta)$. 
To the best of our knowledge, there are only partial results in these directions (see \cite[Theorem~3]{T-2014}, \cite{T-Uniq}).

\smallskip
The article is organized as follows. 
In Section 3, we prove Propositions \ref{prop:non-exist} and \ref{prop:exist-1}.
In Section 4, we prove the existence of solution for $(GEV; \alpha, \beta)$ in the neighbourhood of $(\lambda_1(p), \lambda_1(q))$ provided {\bf (LI)} is satisfied, which becomes the main ingredient in the proof of Proposition \ref{prop:prop}, Part (iv).
In Section 5, we introduce the method of super- and sub-solutions, which is the necessary tool for the proofs in next sections.
Section 6 is devoted to the proof of Proposition \ref{prop:prop}.
In Section 7, we prove the main existence results: Theorems \ref{thm:main2}, \ref{thm:main} and Proposition \ref{prop:special-2}. 
In Section 8, we prove Proposition \ref{prop:equiv}.
In Appendix A, we present a version of the Picone identity for $(p,q)$-Laplacian.
Appendix B is devoted to the proof of non-monotonic property of the first eigenvalue $\lambda_1(p)$ w.r.t. $p > 1$. In Appendix C, we provide an example of the violation of the assumption {\bf (LI)}.

\section{Proofs of Propositions \ref{prop:non-exist} and \ref{prop:exist-1}}

\begin{proof*}{Proposition~\ref{prop:non-exist}} 
Let $\alpha\le \lambda_1(p)$ and $\beta\le \lambda_1(q)$. 
Assume that $u \in \W$ is a non-trivial solution of $(GEV;\alpha,\beta)$. 
Taking $u$ as a test function and using the Poincar\'e inequality, we have 
\begin{align*}
0 \le (\lambda_1(p)-\alpha)\|u\|_p^p &\le \|\nabla u\|_p^p-\alpha\|u\|_p^p \\
&=
\beta\|u\|_q^q-\|\nabla u\|_q^q \le (\beta-\lambda_1(q))\|u\|_q^q\le 0. 
\end{align*}
This chain of inequalities is satisfied if and only if $\alpha=\lambda_1(p)$, $\beta=\lambda_1(q)$ and 
$u$ is the eigenfunction corresponding to $\lambda_1(p)$ and $\lambda_1(q)$ simultaneously.
As a result, our conclusion is shown. 
\end{proof*}

To prove Proposition~\ref{prop:exist-1} we introduce  functional 
$I_{\alpha,\beta}$ on $\W$ by 
\begin{equation}
\label{def:I}
\I(u):=\frac{1}{p}\|\nabla u\|_p^p +\frac{1}{q}\|\nabla u\|_q^q 
-\frac{\alpha}{p}\|u_+\|_p^p -\frac{\beta}{q}\|u_+\|_q^q. 
\end{equation}

\begin{remark}\label{rem:positive} 
It can be shown 
that non-trivial critical points of 
$\I$ correspond to non-negative solutions of $(GEV;\alpha,\beta)$ 
by taking $u_-$ as a test function. 
Moreover, any non-negative solution of $(GEV;\alpha,\beta)$ 
belongs to ${\rm int}\,\C_+$ (see definition \eqref{def:int}).
In fact, if $u$ is a solution of $(GEV;\alpha,\beta)$, 
then $u\in
L^\infty(\Omega)$ by the Moser iteration process (see Appendix A in
\cite{MMT}). 
Hence, the regularity up to the boundary in \cite[Theorem~1]{Lieberman} and \cite[p.~320]{L} ensures that $u\in
C^{1,\beta}_0(\overline{\Omega})$ with some $\beta\in(0,1)$.
Moreover, the strong maximum principle \cite[Theorem~5.4.1]{PS} guarantees that $u>0$ in $\Omega$, since $u \not\equiv 0$ and $u \geq 0$. Therefore, $u$ is a positive
solution of $(GEV;\alpha,\beta)$ and \cite[Theorem~5.5.1]{PS} implies that $\partial u/\partial \nu<0$ on $\partial\Omega$. As a result, $u \in {\rm int}\,\C_+$.
\end{remark}

\begin{proof*}{Proposition~\ref{prop:exist-1}} 
{\bf Case (i):} $\alpha>\lambda_1(p)$ and $0<\beta<\lambda_1(q)$. 
In this case, we note that 
$$
\lambda_1\left(p,\frac{\alpha}{\beta}\right)=\frac{\lambda_1(p)\beta}{\alpha}
<\beta<\lambda_1(q)=\lambda_1(q,1) 
$$
and  
$$
\alpha |u|^{p-2}u +\beta |u|^{q-2}u=
\beta\left(\frac{\alpha}{\beta} |u|^{p-2}u +|u|^{q-2}u\right). 
$$
Thus, our conclusion follows from application of Theorem~\ref{thm:MT-2} to the problem \eqref{eq:MT} with 
$\lambda=\beta$, $m_p=\alpha/\beta$ and $m_q=1$. 

{\bf Case (ii):} $0<\alpha<\lambda_1(p)$ and $\beta>\lambda_1(q)$.
We proceed as above, applying Theorem~\ref{thm:MT-2} to \eqref{eq:MT}
with $\lambda=\alpha$, $m_p=1$ and $m_q=\beta/\alpha$. 

{\bf Case (iii):} $\alpha>\lambda_1(p)$ and $\beta\le 0$. 
By the same argument as in \protect{\cite[Lemma 8]{T-2014}}, 
it can be shown that $\I$ satisfies the Palais--Smale condition. 
Moreover, it is proved in \cite[Theorem 2]{T-2014} that 
for functional $J$ on $\W$ defined by 
$$
J(u)=\frac{1}{p}\|\nabla u\|_p^p +\frac{1}{q}\|\nabla u\|_q^q 
-\frac{\alpha}{p}\|u_+\|_p^p, 
$$
there exist $\delta>0$ and $\rho>0$ such that 
$$
J(u)\ge \delta \quad {\rm provided}\ \|u\|_q=\rho. 
$$
Since $\beta \leq 0$, this implies that 
$\I(u)\geq J(u)\geq \delta$ provided $\|u\|_q=\rho$. 
For the positive eigenfunction $\varphi_p$ corresponding to 
$\lambda_1(p)$ 
and sufficiently large $t>0$, we have 
$$
\I(t\varphi_p)=\frac{t^p}{p}(\lambda_1(p)-\alpha)\|\varphi_p\|_p^p 
+\frac{t^q}{q}\left(\|\nabla \varphi_p\|_q^q
-\beta \|\varphi_p\|_q^q\right)<0,
$$
since $\lambda_1(p)-\alpha<0$ and $p>q$. 
Consequently, by applying the mountain pass theorem, 
 $\I$ has a positive critical value 
(see \cite[Theorem 2]{T-2014} or \cite[Theorem 3]{MT} for the details).

{\bf Case (iv):} $\alpha\le 0$ and $\beta>\lambda_1(q)$. 
In this case, it can be easily shown that 
$\I$ is coercive and bounded from below, due to $q<p$ and 
the inequality 
$$
\I(u)\ge \frac{1}{p}\|\nabla u\|_p^p - \frac{\beta}{q}\|u_+\|_q^q \ge 
\frac{1}{p}\|\nabla u\|_p^p -C\|\nabla u_+\|_p^q,
$$ 
where $C > 0$ is independent of $u \in \W$.
Moreover, $\I$ is weakly lower semi-continuous by the compactness of the embedding of $\W$ to 
$L^p(\Omega)$ and $L^q(\Omega)$, and therefore $\I$ has a global minimizer $u \in \W$ (cf. \cite[Theorem 1.1]{MW}). 
On the other hand, for the positive eigenfunction $\varphi_q$ corresponding to 
$\lambda_1(q)$ 
and sufficiently small $t>0$, we have 
$$
\I (t\varphi_q)=\frac{t^q}{q} (\lambda_1(q)-\beta)\|\varphi_q\|_q^q 
+\frac{t^p}{p} (\|\nabla \varphi_q\|_p^p-\alpha\|\varphi_q\|_p^p)<0, 
$$
whence $I(u) = \min_{\W} \I<0$, and therefore $u$ is a non-trivial solution of $(GEV; \alpha, \beta)$.
\end{proof*}

\section{Existence of solution for $\alpha = \lambda_1(p) + \varepsilon$, 
$\beta = \lambda_1(q) + \varepsilon$ under the assumption {\bf (LI)}}

We define the energy functional corresponding to $(GEV;\alpha,\beta)$ by
$$
E_{\alpha, \beta}(u) = \frac{1}{p} H_\alpha(u) + \frac{1}{q} G_\beta(u),
$$
where for further simplicity we denote
\begin{align*}
H_\alpha(u) &:= \int_\Omega |\nabla u|^p \, dx - \alpha \int_\Omega |u|^p \,dx, \\
G_\beta(u) &:= \int_\Omega |\nabla u|^q \, dx - \beta \int_\Omega |u|^q \,dx.
\end{align*}
Note that $E_{\alpha, \beta} \in C^1(\W, \mathbb{R})$. 
We introduce the so-called \textit{Nehari manifold}
$$
\mathcal{N}_{\alpha, \beta} := \left\{ u \in \W \setminus \{ 0 \}:  \left< E_{\alpha, \beta}'(u), u \right> =  H_\alpha(u) + G_\beta(u) = 0 \right\}.
$$

\begin{proposition}
\label{prop:minpoint}
Let $u \in \W$. If $H_\alpha(u) \cdot G_\beta(u) < 0$, then there exists a unique extremum point $t(u) > 0$ of $E_{\alpha, \beta}(t u)$ w.r.t. $t > 0$, and $t(u) u \in \mathcal{N}_{\alpha, \beta}$.
In particular, if
\begin{equation}
\label{eq:prop11}
G_\beta(u) < 0<H_\alpha(u),
\end{equation}
then $t(u)$ is the unique minimum point of  $E_{\alpha, \beta}(t u)$ w.r.t. $t > 0$, and $E_{\alpha, \beta}(t(u) u) < 0$.
\end{proposition}
\begin{proof}
Fix some non-trivial function $u \in \W$ and consider the fibered functional corresponding to $E_{\alpha, \beta}(u)$:
\begin{align*}
E_{\alpha, \beta}(t u) 
&= \frac{1}{p} H_\alpha(t u) + \frac{1}{q} G_\beta(t u) \\
&= \frac{t^p}{p} H_\alpha(u) + \frac{t^q}{q} G_\beta(u), \quad t > 0.
\end{align*}
Under the assumption $H_\alpha(u) \cdot G_\beta(u) < 0$ the equation
$$
\frac{d}{dt} E_{\alpha, \beta}(t u) = t^{p-1} H_\alpha(u) + t^{q-1} G_\beta(u) = 0, \quad t > 0,
$$
is satisfied for unique $t > 0$ given by
$$
t = t(u) = \left(\frac{-G_\beta(u)}{H_\alpha(u)} \right)^{\frac{1}{p-q}} > 0.
$$
This implies that 
$$
\left< E'_{\alpha, \beta}(t(u) u), t(u)u \right> = t(u) \cdot \left.\frac{d}{dt} E_{\alpha, \beta}(t u)\right|_{t = t(u)} = 0,
$$
and hence $t(u) u \in \mathcal{N}_{\alpha, \beta}$. Moreover, recalling that $q < p$, if \eqref{eq:prop11} holds, then
$$
E_{\alpha, \beta}(t(u) u) = \frac{1}{p} H_\alpha(t(u) u) + \frac{1}{q} G_\beta(t(u) u) = \frac{p-q}{p q} G_\beta(t(u) u) < 0,
$$
and
\begin{align*}
&\left.\frac{d^2}{dt^2} E_{\alpha, \beta}(t u)\right|_{t = t(u)} = (p-1) t(u)^{p-2} H_\alpha(u) + (q-1) t(u)^{q-2} G_\beta(u) \\
&= \frac{1}{t(u)^2} \left( (p-1) H_\alpha(t(u) u) + (q-1) G_\beta(t(u) u) \right) 
= \frac{q - p}{t(u)^2} \, G_\beta(t(u) u) > 0,
\end{align*}
which implies that $t(u)$ is a minimum point of $E_{\alpha, \beta}(t u)$ w.r.t. $t > 0$.
\end{proof}

\begin{lemma}
\label{lemm:nonempty}
Assume that {\bf (LI)} is satisfied. Then there exists $\varepsilon_0 > 0$ such that $\mathcal{N}_{\lambda_1(p)+\varepsilon, \lambda_1(q)+\varepsilon} \neq \emptyset$ for all $\varepsilon \in (0, \varepsilon_0)$.
Moreover, there exists $u \in \mathcal{N}_{\lambda_1(p)+\varepsilon, \lambda_1(q)+\varepsilon}$, such that $E_{\lambda_1(p)+\varepsilon, \lambda_1(q)+\varepsilon}(u) < 0$.
\end{lemma}
\begin{proof}
Since  {\bf (LI)} is satisfied and due to the simplicity of $\lambda_1(p)$, we have $H_{\lambda_1(p)}(\varphi_q) \neq 0$, which yields $H_{\lambda_1(p)}(\varphi_q) > 0$. At the same time, $G_{\lambda_1(q)}(\varphi_q) = 0$ by the definition of $\lambda_1(q)$. Hence, there exists sufficiently small $\varepsilon_0 > 0$ such that for all $\varepsilon \in (0, \varepsilon_0)$ it still holds $H_{\lambda_1(p)+\varepsilon}(\varphi_q) > 0$. Moreover $G_{\lambda_1(q)+\varepsilon}(\varphi_q) < 0$.
Applying now Proposition \ref{prop:minpoint} we get the desired results.
\end{proof}

\begin{lemma}
\label{lemm:crit}
Assume that $H_\alpha(u) \neq 0$ or $G_\beta(u)\not=0$.
If $u \in \mathcal{N}_{\alpha, \beta}$ is a critical point of $E_{\alpha, \beta}$ on $\mathcal{N}_{\alpha, \beta}$, then $u$ is a critical
 point of $E_{\alpha, \beta}$ on $\W$.
\end{lemma}
\begin{proof}
Let $u \in \mathcal{N}_{\alpha, \beta}$ be a critical point of $E_{\alpha, \beta}$ on $\mathcal{N}_{\alpha, \beta}$. 
Since we are assuming that $H_\alpha(u) \neq 0$ 
or $G_\beta(u)\not=0$, $u$ satisfies 
$$
\left< \left(H_\alpha(u) + G_\beta(u)\right)', u \right>
= p H_\alpha(u) + q G_\beta(u) 
=(p-q)H_\alpha(u)=(q-p)G_\beta(u) \neq 0, 
$$
where we used the fact that $H_\alpha(u) + G_\beta(u) = 0$ for $u \in \mathcal{N}_{\alpha,\beta}$. 
This implies that $\left(H_\alpha(u) + G_\beta(u)\right)' \neq 0$ in $W_0^{1,p}(\Omega)^*$. 
Due to the Lagrange multiplier rule (see, e.g. \cite[Theorem 48.B and Corollary 48.10]{zeidler}), there exists $\mu\in\mathbb{R}$ such that 
$$
\left< E_{\alpha, \beta}'(u), \xi \right> = \mu 
\left< \left(H_\alpha(u) + G_\beta(u)\right)', \xi \right>
$$
for each $\xi \in \W$.
Taking $\xi = u$  we get
$$
0= \left< E_{\alpha, \beta}'(u), u \right> = 
\mu \left( p H_\alpha(u) + q G_\beta(u) \right)
=\mu(p-q)H_\alpha(u)=\mu(q-p)G_\beta(u),
$$
since $H_\alpha(u) + G_\beta(u) = 0$. 
Therefore $\mu= 0$ and
$$
\left< E_{\alpha, \beta}'(u), \xi \right> = 0 \quad \mbox{for all } \xi \in \W,
$$
i.e. $u$ is a critical point of $E_{\alpha, \beta}$ on $\W$.
\end{proof}

If $u\in \mathcal{N}_{\lambda_1(p) + \varepsilon, \lambda_1(q) + \varepsilon}$ 
is a minimizer of 
$E_{\lambda_1(p) + \varepsilon, \lambda_1(q) + \varepsilon}$ on the Nehari manifold $\mathcal{N}_{\lambda_1(p) + \varepsilon, \lambda_1(q) + \varepsilon}$ 
and it satisfies $H_{\lambda_1(p) + \varepsilon}(u) \neq 0$ (or equivalently $G_{\lambda_1(q) + \varepsilon}(u) \neq 0$), then $u$ is a critical point of $E_{\lambda_1(p) + \varepsilon, \lambda_1(q) + \varepsilon}$ by Lemma \ref{lemm:crit}, i.e. $u$ is a solution of 
$(GEV;\lambda_1(p) + \varepsilon, \lambda_1(q) + \varepsilon)$.

Since Lemma \ref{lemm:nonempty} implies the existence of $\varepsilon_0 > 0$ such that $\mathcal{N}_{\lambda_1(p) + \varepsilon, \lambda_1(q) + \varepsilon} \neq \emptyset$ for every $\varepsilon \in (0, \varepsilon_0)$, we can find a corresponding minimization sequence $\{ u_k^\varepsilon \}_{k=1}^{\infty} \in \mathcal{N}_{\lambda_1(p) + \varepsilon, \lambda_1(q) + \varepsilon}$, namely, 
$$
E_{\lambda_1(p) + \varepsilon, \lambda_1(q) + \varepsilon}(u_k^\varepsilon) 
\to \inf\left\{
E_{\lambda_1(p) + \varepsilon, \lambda_1(q) + \varepsilon}(u)\,:\, 
u\in 
\mathcal{N}_{\lambda_1(p) + \varepsilon, \lambda_1(q) + \varepsilon}\right\}
=:M_\varepsilon 
$$
as $k\to \infty$. 
The following result states that this minimization sequence is bounded for any sufficiently small $\varepsilon > 0$, and so $M_\varepsilon>-\infty$ holds, 
since $E_{\lambda_1(p) + \varepsilon, \lambda_1(q) + \varepsilon}$ 
is bounded on bounded sets.

\begin{lemma}\label{lem:bdd-below-on-Nehari} 
Assume that {\bf (LI)} is satisfied.
Then there exist $\varepsilon_1 > 0$ and $C > 0$ such that 
$\|\nabla u_k^\varepsilon\|_p \leq C$ for all $k \in \mathbb{N}$ and $\varepsilon \in (0, \varepsilon_1)$. 
\end{lemma}
\begin{proof}
Notice that for $\varepsilon \in (0, \varepsilon_0)$ Lemma \ref{lemm:nonempty} implies
$M_\varepsilon < 0$. Hence, considering sufficiently large $k \in \mathbb{N}$, we may assume that $G_{\lambda_1(q)+\varepsilon}(u_k^\varepsilon) < 0$ by noting that for $u_k^\varepsilon \in \mathcal{N}_{\lambda_1(p) + \varepsilon, \lambda_1(q) + \varepsilon}$ it holds
\begin{equation}
\label{eq:energy<0}
\frac{p-q}{p q} G_{\lambda_1(q)+\varepsilon}(u_k^\varepsilon) = E_{\lambda_1(p)+\varepsilon, \lambda_1(q)+\varepsilon}(u_k^\varepsilon)
\end{equation}
with $q < p$. Consequently, we also get that $H_{\lambda_1(p)+\varepsilon}(u_k^\varepsilon) > 0$ for such $k \in \mathbb{N}$ and $\varepsilon \in (0, \varepsilon_0)$.

Suppose the assertion of the lemma is false. Then for any $m \in \mathbb{N}$ there exist $\varepsilon(m) \in (0, 1/m)$ and $k(m) \in \mathbb{N}$ such that for $u_m := u_{k(m)}^{\varepsilon(m)} \in \mathcal{N}_{\lambda_1(p) + \varepsilon(m), \lambda_1(q) + \varepsilon(m)}$ it holds
$\|\nabla u_m\|_{p} > m$.
Consider the normalized sequence $\{ v_m \}_{m=1}^{\infty}$, such that $u_m = t_m v_m$, $t_m = \|\nabla u_m\|_{p} > m$ and $\|\nabla v_m\|_{p} = 1$.
Then the Eberlein-Shmulyan theorem and the Sobolev embedding theorem imply the existence of a subsequence of $\{v_m\}_{m = 1}^{\infty}$ (which we denote again $\{v_m\}_{m=1}^{\infty}$) and $v^* \in \W$ such that
\begin{align*}
&v_m \rightharpoonup v^* ~\text{ weakly in }~ \W ~\text{ and }~ W_0^{1,q} ~\text{ as $m \to \infty$}, \\
&v_m \rightarrow v^* ~\text{ strongly in }~ L^p(\Omega) ~\text{ and } ~ L^{q}(\Omega) ~\text{ as $m \to \infty$}.
\end{align*}
Moreover, by weakly lower semicontinuity of the norms of $\W$ and $W_0^{1,q}$ we have
\begin{equation}
\label{eq:wls}
\|\nabla v^*\|_p \leq \liminf_{m \to \infty} \|\nabla v_m \|_p, \quad
\|\nabla v^*\|_q \leq \liminf_{m \to \infty} \|\nabla v_m\|_q.
\end{equation}
Since $H_{\lambda_1(p) + \varepsilon(m)}(u_m) = -
G_{\lambda_1(q) + \varepsilon(m)}(u_m)$ for all $m \in \mathbb{N}$, we have
$$
t_m^{p-q} \left|H_{\lambda_1(p) + \varepsilon(m)}(v_m) \right| =
\left| G_{\lambda_1(q) + \varepsilon(m)}(v_m) \right| \le C_1 < +\infty
$$
for some constant $C_1$ uniformly w.r.t. $m \in \mathbb{N}$, 
because $G_{\lambda_1(q) + \varepsilon}$ is bounded on bounded sets and 
$\varepsilon(m)\to 0$. Therefore, taking into account that $t_m \to \infty$, we conclude that $H_{\lambda_1(p) + \varepsilon(m)}(v_m)\to 0$ as $m \to \infty$. Using this fact, \eqref{eq:wls}, and recalling that $G_{\lambda_1(q)+\varepsilon(m)}(v_m)<0$ for all $m \in \mathbb{N}$, we deduce
\begin{align}
\label{eq:conv1}
H_{\lambda_1(p)}(v^*) &\leq \liminf_{m \to \infty} H_{\lambda_1(p) + \varepsilon(m)}(v_m) = 0, \\
\label{eq:conv2}
G_{\lambda_1(q)}(v^*) &\leq \liminf_{m \to \infty} G_{\lambda_1(q) + \varepsilon(m)}(v_m) \leq 0.
\end{align}
Noting that $v_m\to v^*$ in $L^p(\Omega)$ and 
$H_{\lambda_1(p) + \varepsilon(m)}(v_m)= 
1-(\lambda_1(p) + \varepsilon(m))\|v_m\|_p^p$, we get
\begin{align}
\notag
\lambda_1(p) \int_\Omega |v^*|^p \,dx  
&= \limsup_{m \to \infty} (\lambda_1(p) + \varepsilon(m)) \int_\Omega |v_m|^p \,dx \\
&= 1-\liminf_{m\to \infty}H_{\lambda_1(p) + \varepsilon(m)}(v_m)
=1,
\label{eq:nonzero}
\end{align}
which implies that $v^* \not\equiv 0$.
At the same time, in view of \eqref{eq:conv1} and \eqref{eq:conv2}, the Poincar\'e inequality yields
\begin{align*}
H_{\lambda_1(p)}(v^*) = 0 
\quad \text{and} \quad
G_{\lambda_1(q)}(v^*) = 0,
\end{align*}
and therefore from the simplicity of the first eigenvalues $\lambda_1(p)$ and $\lambda_1(q)$ we must have $|v^*| = \varphi_p/\|\nabla \varphi_p\|$ 
and $|v^*| = \varphi_q/\|\nabla \varphi_q\|$ simultaneously. However, it contradicts {\bf (LI)}.
\end{proof}

From Lemma \ref{lem:bdd-below-on-Nehari} it follows that there exist non-trivial weak limits $u_0^\varepsilon \in \W$ of the corresponding minimization subsequences $\{ u_k^\varepsilon \}_{k=1}^{\infty} \in \mathcal{N}_{\lambda_1(p) + \varepsilon, \lambda_1(q) + \varepsilon}$ for any $\varepsilon \in (0, \varepsilon_1)$.
\begin{lemma}
\label{lem:H>0}
Assume that {\bf (LI)} is satisfied.
Then there exists $\varepsilon_2 > 0$ such that
\begin{align}
\label{eq:lem21}
G_{\lambda_1(q) + \varepsilon}(u_0^\varepsilon) < 0 < 
H_{\lambda_1(p) + \varepsilon}(u_0^\varepsilon) 
\end{align}
for all  $\varepsilon \in (0, \varepsilon_2)$.
\end{lemma}
\begin{proof}
Let $\varepsilon_1 > 0$ be given by Lemma \ref{lem:bdd-below-on-Nehari} and $\varepsilon \in(0,\varepsilon_1)$. Note that from \eqref{eq:energy<0} and weakly lower semicontinuity of the norm of $W_0^{1,q}$ it follows that $G_{\lambda_1(q) + \varepsilon}(u_0^\varepsilon) < 0$. 
Therefore, we need to show only that $H_{\lambda_1(p) + \varepsilon}(u_0^\varepsilon) > 0$ for sufficiently small $\varepsilon>0$.

To obtain a contradiction, suppose that for any $m \in \mathbb{N}$ there exists $\varepsilon(m) < 1/m$ such that
$H_{\lambda_1(p) + \varepsilon(m)}(u_0^{\varepsilon(m)}) \leq 0$. 
We consider the normalized sequence $\{ v_m \}_{m=1}^\infty$, where $v_m = u_0^{\varepsilon(m)}/t_m$, $t_m = \|\nabla u_0^{\varepsilon(m)}\|_{p}$ 
and $\|\nabla v_m\|_{p} = 1$.
Therefore, by passing to a subsequence, there exists $v^* \in \W$ such that
\begin{align*}
&v_m \rightharpoonup v^* ~\text{ weakly in }~ \W ~\text{ and }~ W_0^{1,q} ~\text{ as $m \to \infty$ }, \\
&v_m \rightarrow v^* ~\text{ strongly in }~ L^p(\Omega) ~\text{ and } ~ L^{q}(\Omega) ~\text{ as $m \to \infty$ }.
\end{align*}
Using the assumption $H_{\lambda_1(p) + \varepsilon(m)}(u_0^{\varepsilon(m)}) \leq 0$ and weakly lower semicontinuity we have
\begin{align*}
H_{\lambda_1(p)}(v^*) &\leq \liminf_{m \to \infty} H_{\lambda_1(p) + \varepsilon(m)}(v_m) \leq 0, \\
G_{\lambda_1(q)}(v^*) &\leq \liminf_{m \to \infty} G_{\lambda_1(q) + \varepsilon(m)}(v_m) \leq 0.
\end{align*}
Hence, using the Poincar\'e inequality as in the proof of 
Lemma \ref{lem:bdd-below-on-Nehari}, we get a contradiction.
\end{proof}

Now we are able to prove the main result of this section.
\begin{proposition}
\label{prop:main}
Assume that {\bf (LI)} is satisfied. Then 
$u_0^\varepsilon \in \mathcal{N}_{\lambda_1(p) + \varepsilon, \lambda_1(q) + \varepsilon}$ and 
$$M_\varepsilon:=\inf\left\{
E_{\lambda_1(p) + \varepsilon, \lambda_1(q) + \varepsilon}(u)\,:\, 
u\in 
\mathcal{N}_{\lambda_1(p) + \varepsilon, \lambda_1(q) + \varepsilon}\right\}
$$ 
is attained on $u_0^\varepsilon$ for all $\varepsilon \in (0, \varepsilon_2)$, 
where $\varepsilon_2 > 0$ is given by Lemma \ref{lem:H>0}. 
\end{proposition}
\begin{proof}
Fix any $\varepsilon \in (0, \varepsilon_2)$. 
Then there exists a weak limit $u_0^\varepsilon \in \W$ of the minimizing sequence $\{ u_k^\varepsilon\}_{k =1}^\infty \in \mathcal{N}_{\lambda_1(p) + \varepsilon, \lambda_1(q) + \varepsilon}$ and \eqref{eq:lem21} is satisfied. 

Let us show that $u_k^\varepsilon \to u_0^\varepsilon$ strongly in 
$\W$ and $u_0^\varepsilon \in \mathcal{N}_{\lambda_1(p) + \varepsilon, \lambda_1(q) + \varepsilon}$. Indeed, contrary to our claim, 
we suppose that 
$$
\|\nabla u_0^\varepsilon\|_p < \liminf_{k \to \infty} 
\|\nabla u_k^\varepsilon\|_p.
$$ 
Then
\begin{align*}
H_{\lambda_1(p) + \varepsilon}(u_0^\varepsilon) + G_{\lambda_1(q) + \varepsilon}(u_0^\varepsilon) &< \liminf_{k \to \infty} \left( H_{\lambda_1(p) + \varepsilon}(u_k^\varepsilon) + G_{\lambda_1(q) + \varepsilon}(u_k^\varepsilon) \right) = 0,
\end{align*}
which implies that $u_0^\varepsilon \not\in \mathcal{N}_{\lambda_1(p) + \varepsilon, \lambda_1(q) + \varepsilon}$. 
However, according to \eqref{eq:lem21}, the assumptions of Proposition \ref{prop:minpoint} are satisfied. Therefore, there exists a unique minimum point $t(u_0^\varepsilon) \neq 1$ of $E_{\lambda_1(p) + \varepsilon, \lambda_1(q) + \varepsilon}(t u_0^\varepsilon)$ w.r.t. $t > 0$, such that $t(u_0^\varepsilon) u_0^\varepsilon \in \mathcal{N}_{\lambda_1(p) + \varepsilon, \lambda_1(q) + \varepsilon}$. Hence, 
\begin{align*}
M_\varepsilon\le E_{\lambda_1(p) + \varepsilon, \lambda_1(q) + \varepsilon}(t(u_0^\varepsilon) u_0^\varepsilon) 
&< 
E_{\lambda_1(p) + \varepsilon, \lambda_1(q) + \varepsilon}(u_0^\varepsilon) 
\\
&<\liminf_{k \to \infty}
E_{\lambda_1(p) + \varepsilon, \lambda_1(q) + \varepsilon}(u_k^\varepsilon)
=M_\varepsilon,
\end{align*}
which leads to a contradiction. 
Therefore, $u_0^\varepsilon \in \mathcal{N}_{\lambda_1(p) + \varepsilon, \lambda_1(q) + \varepsilon}$ and $u_k^\varepsilon \to u_0^\varepsilon$ strongly in $\W$.
\end{proof}

\begin{lemma}
\label{lemm:well-defined}
Assume that {\bf (LI)} is satisfied. Then  
$(GEV; \lambda_1(p) + \varepsilon, \lambda_1(q) + \varepsilon)$ possesses a positive solution for all $\varepsilon \in (0, \varepsilon_2)$.
\end{lemma}
\begin{proof}

According to Lemma~\ref{lem:H>0} and 
Proposition~\ref{prop:main}, 
$u_0^\varepsilon \in \mathcal{N}_{\lambda_1(p) + \varepsilon, \lambda_1(q) + \varepsilon}$ satisfies \eqref{eq:lem21} and it is a minimizer of 
$E_{\lambda_1(p) + \varepsilon, \lambda_1(q) + \varepsilon}$ 
on $\mathcal{N}_{\lambda_1(p) + \varepsilon, \lambda_1(q) + \varepsilon}$ for all $\varepsilon\in (0,\varepsilon_2)$. 
Since the functional 
$E_{\lambda_1(p) + \varepsilon, \lambda_1(q) + \varepsilon}$ 
is even, we may assume that $u_0^\varepsilon \ge 0$.
Hence, due to Lemma~\ref{lemm:crit} and noting \eqref{eq:lem21}, 
$u_0^\varepsilon$ is a non-trivial and 
non-negative critical point of 
$E_{\lambda_1(p) + \varepsilon, \lambda_1(q) + \varepsilon}$ on $\W$. 
This ensures that $u_0^\varepsilon$ is a positive solution of 
$(GEV;\lambda_1(p) + \varepsilon, \lambda_1(q) + \varepsilon)$ 
(see Remark~\ref{rem:positive}).

\end{proof}

\section{Super- and sub-solutions}

In this section, we introduce the super- and sub-solution method for the problem $(GEV;\alpha, \beta)$.
First we recall the definition of super- and sub-solutions. 
\begin{definition}
A function $u \in \W$ is called a sub-solution 
(resp. super-solution) of $(GEV;\alpha,\beta)$ if 
$u\le 0$ (resp. $\ge 0$) on $\partial\Omega$ and 
$$
\intO \left(|\nabla u|^{p-2}+|\nabla u|^{q-2}\right)\nabla u\nabla \varphi\,dx 
-\intO \left(\alpha|u|^{p-2}u + \beta |u|^{q-2}u \right) \varphi\,dx \le 0 
\quad ({\rm resp. }\ge 0)
$$
for any $\varphi\in \W$ satisfying $\varphi(x)\ge 0$ a.e. $x\in\Omega$.
\end{definition}

In this section, to simplify the notation, we set 
$$
\f(u):=\alpha |u|^{p-2}u+\beta|u|^{q-2}u
$$
Taking any $w, v \in L^\infty(\Omega)$ such that $w\le v$ a.e. in $\Omega$, 
we introduce a truncation
\begin{gather*}
\f^{[v,w]}(x,t):=
\begin{cases}
\f(w(x)) & {\rm if}\ t\ge w(x), \\
\f(t) & {\rm if }\ v(x) < t < w(x), \\
\f(v(x)) & {\rm if}\ t\le v(x),
\end{cases}
\end{gather*}
and define the $C^1$-functional 
\begin{equation}
\label{def of E}
\E^{[v,w]}(u):=\frac{1}{p}\intO |\nabla u|^p\, dx 
+\frac{1}{q}\intO |\nabla u|^q\,dx 
-\intO \int_0^{u(x)}\f^{[v,w]}(x,t)\,dtdx. 
\end{equation}
It is easily seen that $\f^{[v,w]}(x,t)=\f(u(x))$ 
provided $v(x) \le t \le w(x)$. 

\begin{remark}\label{rem:minimizer}
Let $v, w\in L^\infty(\Omega)$ be a sub-solution 
and a super-solution of $(GEV;\alpha,\beta)$, respectively, and they satisfy 
$v\le w$ in $\Omega$. 
It follows from the boundedness of $\f^{[v,w]}$ that 
$\E^{[v,w]}$ is coercive and bounded from below on $\W$ (cf. \cite{MMT}). 
Moreover, it is easy to see that $\E^{[v,w]}$ is weakly lower semi-continuous. 
Hence, by the standard arguments ({\it cf}. \cite[Theorem 1.1]{MW}), 
$\E^{[v,w]}$ has a global minimum point $u \in \W$, which becomes a solution of $(GEV;\alpha,\beta)$.
Moreover, $u\in[v,w]$.
Indeed, since $w$ is a super-solution, taking $(u-w)_+$ as a test function, 
we have 
\begin{align*} 
0&\ge \langle (\E^{[v,w]})'(u), (u-w)_+\rangle -
\langle -\Delta_p w-\Delta_q w, (u-w)_+\rangle 
\\
& \qquad +\intO (\alpha|w|^{p-2}w + \beta|w|^{q-2}w)(u-w)_+\,dx 
\\
&=\int_{u>w} (|\nabla u|^{p-2}\nabla u-|\nabla w|^{p-2}\nabla w)
(\nabla u-\nabla w)\,dx 
\\ 
&\qquad +\int_{u>w} (|\nabla u|^{q-2}\nabla u-|\nabla w|^{q-2}\nabla w)
(\nabla u-\nabla w)\,dx\ge 0,
\end{align*}
where we take into account that $\f^{[v,w]}(x,t)=\alpha|w|^{p-2}w+\beta|w|^{q-2}w$ provided 
$t \ge w(x)$. 
This implies that $u\le w$. 
Similarly, by taking $(u-v)_-$ as test function, 
we see that $u\ge v$ holds. Therefore, 
$\f^{[v,w]}(x,u(x))=\alpha|u|^{p-2}u+\beta|u|^{q-2}u$, whence 
$u$ is a solution of $(GEV;\alpha,\beta)$. 

In particular, 
if a sub-solution $v \geq 0$ and $u$ is not-trivial, 
then it is known that $u\in {\rm int}\,\C_+$ 
(see Remark~\ref{rem:positive}). 
\end{remark}

\begin{lemma}\label{lem:ss-method} 
Assume that $\beta>\lambda_1(q)$ and 
$w\in {\rm int}\,C^1(\overline{\Omega})_+$ is a positive super-solution of 
$(GEV;\alpha,\beta)$. 
Then $\min_{\W} \E^{[0,w]}<0$ holds, and hence 
$(GEV;\alpha,\beta)$ has a positive solution belonging to 
${\rm int}\,\C_+$. 
\end{lemma} 
\begin{proof}
Let $\beta>\lambda_1(q)$ and 
$w\in {\rm int}\,C^1(\overline{\Omega})_+$ be a positive super-solution of 
$(GEV;\alpha,\beta)$. 
Recall that $\E^{[0,w]}$ has a global minimum point as stated in 
Remark~\ref{rem:minimizer}. 

Since $w\in {\rm int}\,C^1(\overline{\Omega})_+$, 
for sufficiently small $t>0$ we have $w-t\varphi_q \ge 0$ in $\Omega$. 
This implies that 
$\f^{[0,w]}(x, t\varphi_q)=\alpha t^{q-1} \varphi_q^{p-1}+\beta t^{p-1} \varphi_q^{q-1}$. 
Hence, for sufficiently small $t>0$, we obtain 
$$
\E^{[0,w]}(t\varphi_q)=\frac{t^p}{p}
(\|\nabla \varphi_q\|^p-\alpha\|\varphi_q\|_p^p) 
-\frac{t^q}{q}(\beta-\lambda_1(q)) \|\varphi_q\|_q^q. 
$$
Recalling that $q<p$ and $\beta-\lambda_1(q)>0$, we see that 
$\E^{[0,w]}(t\varphi_q)<0$ for sufficiently small $t>0$, 
whence $\min_{\W}\E^{[0,w]}<0$. 
Therefore, $\E^{[0,w]}$ has a non-trivial critical point, and  
our conclusion follows (see Remark~\ref{rem:minimizer}). 
\end{proof}

\section{Properties of $\lambda^*(s)$} 
In this section we prove Proposition \ref{prop:prop}.

\begin{proof*}{Proposition \ref{prop:prop}} \textbf{Part (i).}
Fix any $s \in \mathbb{R}$ and let $u \in \W$ be a positive solution of $(GEV;\lambda+s, \lambda)$ for some $\lambda \in \mathbb{R}$. 
Then $u \in {\rm int}\, \C_+$ by
Remark~\ref{rem:positive}. 
Choose any $\varphi \in {\rm int}\C_+$. Then, $\varphi/u \in L^\infty(\Omega)$, 
and hence we can take 
$$
\xi = \frac{\varphi^p}{u^{p-1} + u^{q-1}} \in \W 
$$
as a test function. 
Therefore, from Proposition \ref{picone} there follows the existence of $\rho > 0$ independent of $u$ and $\lambda$ such that
\begin{align*}
&\int_\Omega |\nabla u|^{p-2}  \nabla u \nabla \left( \frac{\varphi^p}{u^{p-1} + u^{q-1}} \right)\, dx  + \int_\Omega|\nabla u|^{q-2} \nabla u \nabla \left( \frac{\varphi^p}{u^{p-1} + u^{q-1}} \right)\, dx \\
&= \lambda \int_\Omega \varphi^p \, dx + s \int_\Omega \frac{u^{p-1} \varphi^p}{u^{p-1} + u^{q-1}} \, dx \leq \frac{1}{\rho} \left( \int_\Omega |\nabla \varphi|^p \,dx + 
\int_\Omega |\nabla (\varphi^{p/q})|^q \,dx \right).
\end{align*}
Combining this inequality with the estimation
$$
s \int_\Omega \frac{u^{p-1} \varphi^p}{u^{p-1} + u^{q-1}} \, dx \geq 
\min\left\{ 0, s \int_\Omega \varphi^p \, dx \right\}, \quad s \in \mathbb{R},
$$
we conclude that
\begin{equation}
\label{eq:picone:3}
\lambda \int_\Omega \varphi^p \,dx + \min\left\{ 0, s \int_\Omega \varphi^p \, dx \right\} \leq 
\frac{1}{\rho} \left( \int_\Omega |\nabla \varphi|^p \,dx + 
\int_\Omega |\nabla (\varphi^{p/q})|^q \,dx \right).
\end{equation}
Since $\int_\Omega \varphi^p \,dx$, $\int_\Omega |\nabla \varphi|^p \,dx$, $\int_\Omega |\nabla (\varphi^{p/q})|^q \,dx$ and $\rho$ 
are positive constants independent of $u$ and $\lambda$, 
$\lambda$ satisfying \eqref{eq:picone:3} is bounded from above. Therefore,  
$\lambda^*(s) < +\infty$, which completes the proof of Part (i).

\smallskip
\par\noindent
\textbf{Part (iv).}
Assume first that {\bf (LI)} holds. Then Lemma \ref{lemm:well-defined} implies that $(GEV; \lambda_1(p) + \varepsilon, \lambda_1(q) + \varepsilon)$ possesses a positive solution for sufficiently small $\varepsilon > 0$.
Noting that
$(\lambda_1(p) + \varepsilon, \lambda_1(q) + \varepsilon)
=(\lambda_1(q) + \varepsilon+s^*,\lambda_1(q) + \varepsilon)$, 
by definition of $\lambda^*(s^*)$ 
we have $\lambda^*(s^*)\ge \lambda_1(q)+\varepsilon$, 
and so $\lambda^*(s^*)+s^*\ge \lambda_1(q)+\varepsilon+s^*
=\lambda_1(p)+\varepsilon$, which is the desired conclusion.

Assume now that {\bf (LI)} is violated, i.e. $\varphi_p \equiv k \varphi_q$ in $\Omega$ for some $k \neq 0$. Let $u$ be a positive weak solution of $(GEV; \alpha, \beta)$ for some $\alpha, \beta \in \mathbb{R}$. Then, due to the regularity of $\varphi_p$ and $u$ (see Remark \ref{rem:positive}), the classical Picone identity \cite{Alleg} implies
\begin{equation}
\label{eq:equiv:pic1}
\int_\Omega |\nabla u|^{p-2}  \nabla u \nabla \left( \frac{\varphi_p^p}{u^{p-1}} \right)\, dx \leq 
\int_\Omega |\nabla \varphi_p|^p \, dx = \lambda_1(p) \int_\Omega \varphi_p^p \, dx.
\end{equation}
At the same time, generalized Picone's identity from \cite[Lemma~1, p.~536]{ilcomp} yields
\begin{align}
\notag
\int_\Omega |\nabla u|^{q-2}  \nabla u \nabla \left( \frac{\varphi_p^p}{u^{p-1}} \right)\, dx &\leq
\int_\Omega |\nabla \varphi_p|^{q-2} \nabla \varphi_p \nabla \left( \frac{\varphi_p^{p-q+1}}{u^{p-q}} \right) \, dx \\
\label{eq:equiv:pic2}
&= \lambda_1(q) \int_\Omega \varphi_p^{p} u^{q-p} \, dx,
\end{align}
where the last equality is valid because $\varphi_p$ is an eigenfunction of $-\Delta_q$, by assumption.

Hence, using \eqref{eq:equiv:pic1} and \eqref{eq:equiv:pic2}, we obtain for the solution $u$ of $(GEV; \alpha, \beta)$ the following inequality:
\begin{align*}
&\int_\Omega |\nabla u|^{p-2}  \nabla u \nabla \left( \frac{\varphi_p^p}{u^{p-1}} \right)\, dx
+
\int_\Omega |\nabla u|^{q-2}  \nabla u \nabla \left( \frac{\varphi_p^p}{u^{p-1}} \right)\, dx
\\ 
&=
\alpha \int_\Omega \varphi_p^p \, dx 
+
\beta \int_\Omega \varphi_p^p u^{q-p} \, dx 
\leq 
\lambda_1(p) \int_\Omega \varphi_p^p \, dx + \lambda_1(q) \int_\Omega \varphi_p^{p} u^{q-p} \, dx,
\end{align*}
which is impossible if $\alpha > \lambda_1(p)$ and $\beta > \lambda_1(q)$ simultaneously, and the proof is complete.

\smallskip
\par\noindent
\textbf{Part (ii).}
Assume that $s \neq s^*$. Then taking $\alpha = \lambda + s$ and $\beta = \lambda$, Proposition \ref{prop:exist-1} implies that $\lambda^*(s) + s \geq \lambda_1(p)$ and $\lambda^*(s) \geq \lambda_1(q)$. If now $s = s^*$ and $\lambda_1(p)$ and $\lambda_1(q)$ have the same eigenspace, i.e. there exists $k \neq 0$ such that $\varphi_p \equiv k \varphi_q$ in $\Omega$, then from Proposition \ref{prop:non-exist} it follows that $(GEV; \lambda_1(p), \lambda_1(q))$ possesses a positive solution, i.e. $\lambda^*(s^*) + s^* \geq \lambda_1(p)$ and $\lambda^*(s^*) \geq \lambda_1(q)$. Finally, if $\lambda_1(p)$ and $\lambda_1(q)$ have different eigenspaces, that is, {\bf (LI)} is satisfied, then Part (iv) of Proposition \ref{prop:prop} yields the desired result.

\smallskip
\par\noindent
\textbf{Part (vi).}
Let $s<s'$. 
Part (ii) of Proposition \ref{prop:prop} implies that 
$\lambda^*(s), \lambda^*(s') \geq \lambda_1(q)$. 
Thus, in order to prove $\lambda^*(s)\ge \lambda^*(s')$, 
it is sufficient to consider only the case 
$\lambda^*(s')>\lambda_1(q)$. 

Fix any $\varepsilon>0$ such that 
$\lambda^*(s')-\varepsilon>\lambda_1(q)$. 
Then, by the definition of $\lambda^*(s')$, 
there exists $\mu$ satisfying 
$\lambda^*(s')>\mu>\lambda^*(s')-\varepsilon$ 
such that $(GEV;\mu+s',\mu)$ has a positive solution 
$w_\mu\in {\rm int}\,\C_+$. 
It is easy to see that $w_\mu$ is a positive 
super-solution of $(GEV;\mu+s,\mu)$, since $s < s'$. Hence, 
Lemma~\ref{lem:ss-method} ensures the existence of a positive solution of 
$(GEV;\mu+s,\mu)$ (note $\mu>\lambda^*(s')-\varepsilon>\lambda_1(q)$). 
Hence, $\lambda^*(s)\ge \mu(>\lambda^*(s')-\varepsilon)$. 
Since $\varepsilon$ is arbitrary, we have 
$\lambda^*(s) \geq \lambda^*(s')$. 

Next, we show that
$\lambda^*(s)+s\le \lambda^*(s')+s'$ for $s < s'$.
If $\lambda^*(s)+s-s'\le \lambda_1(q)$, then 
$\lambda^*(s)+s\le \lambda_1(q)+s'\le \lambda^*(s')+s'$, due to the fact that  
$\lambda_1(q) \leq \lambda^*(s')$. 
So, we may suppose that 
$\lambda^*(s)+s-s'>\lambda_1(q)$. 
Fix any $\varepsilon>0$ such that 
$\lambda^*(s)+s-s'-\varepsilon>\lambda_1(q)$. 
By the definition of $\lambda^*(s)$, there exists 
$\mu>\lambda^*(s)-\varepsilon$ such that 
$(GEV;\mu+s,\mu)$ has a positive solution $w_\mu$. 
Putting $\beta=\mu+s-s'$, $w_\mu$ is the positive solution of $(GEV;\beta+s',\beta+s'-s)$. 
Noting that $\beta+s'-s>\beta$, $w_\mu$ is a positive 
super-solution of $(GEV;\beta+s',\beta)$. 
Since $\beta>\lambda^*(s)+s-s'-\varepsilon>\lambda_1(q)$, 
by the same argument above, 
we get $\lambda^*(s') \geq \lambda^*(s)+s-s'$, whence 
$\lambda^*(s)+s \leq \lambda^*(s')+s'$.

\smallskip
\par\noindent
\textbf{Part (iii).}
Assume first that {\bf (LI)} doesn't hold. 
Then $s^* = s^*_+$ and, due to Part (iv) of Proposition \ref{prop:prop}, $\lambda^*(s^*) \leq \lambda_1(q)$. At the same time, $\lambda^*(s) \geq \lambda_1(q)$ for all $s \in \mathbb{R}$ by Part (ii). Hence, $\lambda^*(s^*) = \lambda_1(q)$ and noting that $\lambda^*(s)$ is non-increasing by Part (vi) we get the desired result.

Let now {\bf (LI)} hold and suppose, by contradiction, that there exists $s \geq s^*_+$ such that $\lambda^*(s) > \lambda_1(q)$.
Since $\lambda^*(s^*) + s^* > \lambda_1(p)$ by Part (iv) of Proposition \ref{prop:prop}, using Part (vi) and recalling that $s$, we get 
$$
\lambda^*(s)+s \geq \lambda^*(s^*) + s^* > \lambda_1(p).
$$
By definition of $\lambda^*(s)$, for any $\varepsilon_0 > 0$ there exists $\varepsilon \in [0, \varepsilon_0)$ such that $(GEV;\lambda^*(s) + s - \varepsilon,\lambda^*(s) - \varepsilon)$
possesses a positive solution. Let us take $\varepsilon_0$ small enough to satisfy
\begin{equation}
\label{eq:partii2}
\lambda^*(s) + s - \varepsilon_0 > \lambda_1(p), \quad 
\lambda^*(s) - \varepsilon_0 > \lambda_1(q),
\end{equation}
and let $u$ be a corresponding solution of $(GEV;\lambda^*(s) + s - \varepsilon,\lambda^*(s) - \varepsilon)$, where $\varepsilon \in [0, \varepsilon_0)$.

Using the Picone identities \eqref{eq:equiv:pic1} and \eqref{eq:equiv:pic2} applied to $\varphi_q$ instead of $\varphi_p$, we obtain the following inequality:
\begin{align}
\notag
&\int_\Omega |\nabla u|^{p-2}  \nabla u \nabla \left( \frac{\varphi_q^p}{u^{p-1}} \right)\, dx
+
\int_\Omega |\nabla u|^{q-2}  \nabla u \nabla \left( \frac{\varphi_q^p}{u^{p-1}} \right)\, dx
\\ 
\notag
&=
(\lambda^*(s) + s - \varepsilon) \int_\Omega \varphi_q^p \, dx 
+
(\lambda^*(s) - \varepsilon) \int_\Omega \varphi_q^p u^{q-p} \, dx \\
\label{eq:partii2-1}
&\leq 
\int_\Omega |\nabla \varphi_q|^p \, dx + \lambda_1(q) \int_\Omega \varphi_q^{p} u^{q-p} \, dx,
\end{align}
On the other hand, since $\varepsilon < \varepsilon_0$, from \eqref{eq:partii2} it follows that
\begin{align}
\notag
(\lambda_1(q) &+ s) \int_\Omega \varphi_q^p \, dx + \lambda_1(q) \int_\Omega \varphi_q^p u^{q-p} \, dx \\
\label{eq:partii2-2}
&< 
(\lambda^*(s) + s - \varepsilon) \int_\Omega \varphi_q^p \, dx 
+
(\lambda^*(s) - \varepsilon) \int_\Omega \varphi_q^p u^{q-p} \, dx
\end{align}
Finally, combining \eqref{eq:partii2-1} with \eqref{eq:partii2-2} we conclude that
$$
s < \frac{\int_\Omega |\nabla \varphi_q|^p \, dx}{\int_\Omega \varphi_q^p \, dx} - \lambda_1(q) = s^*_+,
$$
which contradicts our assumption $s \geq s^*_+$.

\smallskip
\par\noindent
\textbf{Part (v).}
Since $\lambda^*(s)$ is bounded for any $s \in \mathbb{R}$ by Part (i) of Proposition \ref{prop:prop} and non-increasing by Part (vi), for every $s' \in \mathbb{R}$ there exist one-sided limits of $\lambda^*(s)$ and 
\begin{equation}\label{eq:mon:l:1}
\lim_{s \to s'-0} \lambda^*(s) \geq \lambda^*(s') \geq \lim_{s \to s' + 0} \lambda^*(s).
\end{equation}
On the other hand, $\lambda^*(s) + s$ is non-decreasing by Part (vi) of Proposition \ref{prop:prop}, and hence
$$
\lim_{s \to s'-0} (\lambda^*(s) + s) \leq \lambda^*(s') + s' \leq \lim_{s \to s' + 0} (\lambda^*(s) + s),
$$
which yields
\begin{equation}\label{eq:mon:l:2}
\lim_{s \to s'-0} \lambda^*(s) \leq \lambda^*(s') \leq \lim_{s \to s' + 0} \lambda^*(s).
\end{equation}
Combining \eqref{eq:mon:l:1} with \eqref{eq:mon:l:2} we conclude that the one-sided limits are equal to $\lambda^*(s')$, which establishes the desired continuity, due to the arbitrary choice of $s' \in \mathbb{R}$.
\end{proof*}

\section{Proof of Theorems~\ref{thm:main2}, \ref{thm:main} and Proposition~\ref{prop:special-2}}

\begin{proof*}{Theorem~\ref{thm:main2}} 
Note first that from Propositions \ref{prop:non-exist} and \ref{prop:exist-1} it directly follows that if \eqref{eq:condition} is satisfied, then 
$(GEV;\alpha,\beta)$ has at least one positive solution.

Conversely, if $(GEV;\alpha,\beta)$ has at least one positive solution, 
then by the definition of $\lambda^*(s)$, Part (iv) of Proposition \ref{prop:prop},  and Proposition \ref{prop:non-exist}, 
it has to satisfy 
$$
(\alpha,\beta)\in (\lambda_1(p),+\infty)\times(-\infty,\lambda_1(q)] 
\cup (-\infty,\lambda_1(p)]\times(\lambda_1(q),+\infty)\cup\{(\lambda_1(p),\lambda_1(q))\}. 
$$
To prove \eqref{eq:condition}, it is sufficient to show that 
$(\alpha,\beta)\not\in \{\lambda_1(p)\}\times (\lambda_1(q),+\infty)$ 
and $(\alpha,\beta)\not\in (\lambda_1(p),+\infty)\times \{\lambda_1(q)\}$. 
Suppose that 
$(GEV;\alpha,\beta)$ has a positive solution $u$ for 
$\alpha=\lambda_1(p)$ and $\beta \geq \lambda_1(q)$ 
(resp. $\alpha \geq \lambda_1(p)$ and $\beta=\lambda_1(q)$). 
Then, as in the proof of  Proposition \ref{prop:prop}, Part (iv), from \eqref{eq:equiv:pic1} and \eqref{eq:equiv:pic2} we have
\begin{align*}
&\int_\Omega |\nabla u|^{p-2}  \nabla u \nabla \left( \frac{\varphi_p^p}{u^{p-1}} \right)\, dx
+
\int_\Omega |\nabla u|^{q-2}  \nabla u \nabla \left( \frac{\varphi_p^p}{u^{p-1}} \right)\, dx
\\ 
&=
\alpha \int_\Omega \varphi_p^p \, dx 
+
\beta \int_\Omega \varphi_p^p u^{q-p} \, dx 
\leq 
\lambda_1(p) \int_\Omega \varphi_p^p \, dx + \lambda_1(q) \int_\Omega \varphi_p^{p} u^{q-p} \, dx,
\end{align*}
which implies that $\beta = \lambda_1(q)$ (resp. $\alpha = \lambda_1(p)$).
Hence, we get the desired result.
\end{proof*}

\begin{proof*}{Theorem~\ref{thm:main}} 
Consider first the non-existence result. Let $\beta>\lambda^*(\alpha-\beta)$. Then the definition of $\lambda^*(\alpha-\beta)$ implies that 
$(GEV;\alpha,\beta)$ has no positive solutions. 

(i)~ Assume that $\lambda_1(q)<\beta<\lambda^*(s)$ with $s=\alpha-\beta$.
Then, by the definition of $\lambda^*(s)$, there exists $\mu \in (\beta, \lambda^*(s)]$ such that $(GEV;\mu+s, \mu)$ has a positive solution $w \in {\rm int}\, \C_+$ 
(see Remark~\ref{rem:positive}). Moreover, $w$ is a positive super-solution 
of $(GEV;\alpha, \beta) \equiv (GEV;\beta + s, \beta)$, since $\mu>\beta$.
Hence, the assumptions of Lemma~\ref{lem:ss-method} are satisfied, which guarantees the existence of a positive solution of $(GEV;\alpha, \beta)$. 

(ii)~ Assume now that $\lambda_1(p)<\alpha$ and $\beta<\lambda^*(s)$ with $s=\alpha-\beta$.
Note that if $\beta > \lambda_1(q)$, then Part (i) gives the claim. If 
$\beta < \lambda_1(q)$, then Proposition \ref{prop:exist-1} implies the desired result. Therefore, it remains to consider the case $\beta = \lambda_1(q)$.

Let us divide the proof into tree cases:

{\bf Case 1.} $\varphi_q$ satisfies 
$\|\nabla \varphi_q\|_p^p- \alpha \|\varphi_q\|_p^p>0$. 
Note that
$$
\lambda_1 \left(q,\frac{\lambda_1(q)}{\alpha} \right)=\frac{\lambda_1(q) \alpha}{\lambda_1(q)}
= \alpha > \lambda_1(p)=\lambda_1(p,1).
$$
This yields 
\eqref{hyp:MT-3} with $r=q$, $r'=p$, 
$\lambda=\alpha$, $m_p\equiv 1$ and $m_q \equiv \frac{\lambda_1(q)}{\alpha}$. 
Moreover, since $\lambda_1(r,c)$ and $\lambda_1(r,1)=\lambda_1(r)$ 
have the same eigenspace for any constant $c > 0$, namely, 
$\varphi_1(r,c)=t \varphi_1(r,1)=t\varphi_r$ 
for some $t>0$, the hypothesis of Case 1 ensures \eqref{hypo2:MT-3}. Hence,
Theorem~\ref{thm:MT-3} guarantees our conclusion.

{\bf Case 2.} $\varphi_q$ satisfies 
$\|\nabla \varphi_q\|_p^p-\alpha\|\varphi_q\|_p^p<0$. 
Since $\beta < \lambda^*(s)$, by definition of $\lambda^*(s)$ there exists $\mu \in (\beta, \lambda^*(s)]$ such that $(GEV; \mu + s, \mu)$ possesses a positive solution $w \in {\rm int}\,\C_+$. As in the proof of Part (i) it is easy to see that $w$ is a positive super-solution of $(GEV;\alpha, \beta) \equiv (GEV;\beta + s, \beta)$.

Let $E_{\alpha,\beta}^{[0,w]}$ be the functional defined 
by \eqref{def of E} with 
a positive super-solution $w$ and sub-solution $0$.
Since $w$ and $\varphi_q$ belong to ${\rm int}\,\C_+$, 
for sufficiently small $t>0$ we get $t\varphi_q\le w$ in $\Omega$, 
whence 
$f_{\alpha,\beta}^{[0,w]}(x, t\varphi_q)
=\alpha t^{p-1}\varphi_q^{p-1} 
+\beta t^{q-1}\varphi_q^{q-1}$. 
Therefore, noting that $\beta=\lambda_1(q)$, for such small $t>0$ we have 
$$
\E^{[0,w]}(t\varphi_q)=
\frac{t^{p}}{p}\left(\|\nabla \varphi_q\|_p^p
-\alpha\|\varphi_q\|_p^p\right)<0. 
$$
This ensures that $\inf_{\W}\E^{[0,w]}<0$. 
Hence, $(GEV;\alpha,\beta)$ has a positive solution (refer to Remark 5.2).

{\bf Case 3.} $\varphi_q$ satisfies 
$\|\nabla \varphi_q\|_p^p-\alpha\|\varphi_q\|_p^p=0$. 
Similarly to Case 2, we know that 
$\E^{[0,w]}(t\varphi_q)=0$ for sufficiently small 
$t>0$. 
If $\min_{\W}\E^{[0,w]}<0$ holds, then 
$(GEV;\alpha,\beta)$ has a positive solution. 
On the other hand, if $\min_{\W}\E^{[0,w]}=0$, 
then $t\varphi_q$ is a global minimizer 
of $\E^{[0,w]}$, whence 
$t\varphi_q$ is a positive solution of $(GEV;\alpha,\beta)$. 
Consequently, the proof is complete.

\end{proof*}


For the proof of Proposition~\ref{prop:special-2}, 
we prepare two lemmas. 
The following lemma is needed to prove the boundedness of approximate solutions. 
\begin{lemma}\label{lem:bdd-sol} 
Let $u_n$ be a positive solution of 
$(GEV;\alpha_n,\beta_n)$ with $\alpha_n\to \alpha$ and $\beta_n\to \beta$. 
If $\|\nabla u_n\|_p\to \infty$ as $n\to\infty$, 
then $\alpha=\lambda_1(p)$. 
\end{lemma} 
\begin{proof} 
Let $u_n$ be a positive solution of 
$(GEV;\alpha_n,\beta_n)$ with $\alpha_n\to \alpha$, $\beta_n\to \beta$ 
and $\|\nabla u_n\|_p\to \infty$ as $n\to\infty$. 
Setting $w_n:=u_n/\|\nabla u_n\|_p$, we may admit, up to subsequence, that $w_n \to w_0$ weakly in
$\W$ and strongly in $L^p(\Omega)$ and $L^q(\Omega)$ for some $w_0\in\W$. 
By taking
$(w_n-w_0)/\|\nabla u_n\|_p^{p-1}$ as a test function, we obtain
\begin{align*}
0&=\intO |\nabla w_n|^{p-2}\nabla w_n\nabla (w_n-w_0)\,dx
+\frac{1}{\| \nabla u_n\|_p^{p-q}}\intO |\nabla w_n|^{q-2}\nabla w_n\nabla
(w_n-w_0)\,dx
\\
&\qquad -\alpha_n\intO w_n^{p-1}(w_n-w_0)\,dx
-\frac{\beta_n}{\| \nabla u_n\|_p^{p-q}}\intO w_n^{q-1}(w_n-w_0)\,dx
\\
&=\intO |\nabla w_n|^{p-2}\nabla w_n\nabla (w_n-w_0)\,dx +o(1),
\end{align*}
where $o(1)\to 0$ as $n\to\infty$. 
Due to the $(S_+)$ property of
$-\Delta_p$ (cf. \cite[Definition 5.8.31 and Lemma 5.9.14]{drabekmilota}), 
this implies that $w_n\to w_0$ strongly in
$\W$. Then, for any $\varphi\in \W$, by taking
$\varphi/\|\nabla u_n\|_p^{p-1}$ as test function we have
\begin{align*}
0 &=\intO |\nabla w_n|^{p-2}\nabla w_n\nabla \varphi\,dx
+\frac{1}{\|\nabla u_n\|_p^{p-q}} \intO |\nabla w_n|^{q-2}\nabla w_n\nabla
\varphi\,dx
\\
&\qquad -\alpha_n\intO w_n^{p-1} \varphi\,dx
-\frac{\beta_n}{\|\nabla u_n\|_p^{p-q}}\intO w_n^{q-1}\varphi\,dx.
\end{align*}
Letting $n\to\infty$ we conclude that $w_0$ is a non-negative,
non-trivial solution of $(EV;p,\alpha)$ 
(note $w_0\ge0$ and $\|\nabla w_0\|_p=1$). 
According to the strong
maximum principle (see Remark~\ref{rem:positive}), we have 
$w_0>0$ in $\Omega$.
This yields that $w_0$ is a positive eigenfunction corresponding 
to $\alpha$ and $\alpha=\lambda_1(p)$, since any eigenvalue other than $\lambda_1(p)$ 
has no positive eigenfunctions.
\end{proof} 

\begin{lemma}\label{lem:bdd-small} 
If $u$ is a positive solution of 
$(GEV;\alpha,\beta)$, then 
$$
\intO |\nabla \varphi|^q|\nabla u|^{p-q}\,dx+\intO |\nabla \varphi|^q\,dx 
\ge \intO \left(\alpha u^{p-q}+\beta\right)\varphi^q\,dx 
$$
for every $\varphi\in {\rm int}\,\C_+$. 
\end{lemma}
\begin{proof} 
Let $u$ be a positive solution of $(GEV;\alpha,\beta)$. 
Then, $u\in {\rm int}\,\C_+$ (see Remark~\ref{rem:positive}). 
Choose any $\varphi\in {\rm int}\,\C_+$. Then, $\varphi/u \in L^\infty(\Omega)$, 
and hence we can take 
$$
\xi = \frac{\varphi^q}{u^{q-1}} \in \W 
$$
as a test function. By the similar estimation as in the proof of 
Proposition~\ref{picone}, we have 
\begin{align}
\notag
& |\nabla u|^{p-2} \nabla u \nabla \left( \frac{\varphi^q}{u^{q-1}} \right)  
= q |\nabla u|^{p-2} \nabla u \nabla \varphi 
\left(\frac{\varphi}{u}\right)^{q-1} - 
(q-1)|\nabla u|^{p} \left(\frac{\varphi}{u}\right)^{q} 
\\
\label{eq:pic_id_1}
& \le q |\nabla u|^{p-1} |\nabla \varphi| 
\left(\frac{\varphi}{u}\right)^{q-1} - 
(q-1)|\nabla u|^{p} \left(\frac{\varphi}{u}\right)^{q} 
\leq|\nabla \varphi|^q|\nabla u|^{p-q}
\end{align}
in $\Omega$, where we use 
the standard Young's inequality 
$$
ab\le \frac{a^q}{q}+\frac{(q-1) b^{q/(q-1)}}{q}
$$
with $a=|\nabla \varphi| |\nabla u|^{p-1-d}$, 
$b=(\varphi/u)^{q-1}|\nabla u|^d$ 
and $d=(q-1)p/q=p-p/q$. 

At the same time, the standard Picone identity \cite{Alleg} implies
\begin{equation}
\label{eq:pic_id_2}
|\nabla u|^{q-2} \nabla u \nabla \left( \frac{\varphi^q}{u^{q-1}} \right)  
\leq|\nabla \varphi|^q \quad \mbox{in } \Omega.
\end{equation}

Applying now estimations \eqref{eq:pic_id_1} and \eqref{eq:pic_id_2} to the definition of a weak solution, we obtain the desired result.
\end{proof}

\begin{proof*}{Proposition~\ref{prop:special-2}} 
\textbf{Part (i).}
Put $\alpha=\lambda^*(s)+s>\lambda_1(p)$ and 
$\beta=\lambda^*(s)>\lambda_1(q)$ for some $s\in\mathbb{R}$. 
By the definition of $\lambda^*(s)$, there exists $\beta_n>\lambda_1(q)$ 
such that $\beta_n\to \beta=\lambda^*(s)$ and 
$(GEV;\alpha_n,\beta_n)$ has a positive solution 
$u_n$, where $\alpha_n=\beta_n+s$. 
Since $\alpha_n\to \beta+s=\lambda^*(s)+s>\lambda_1(p)$, 
Lemma~\ref{lem:bdd-sol} guarantees the boundedness of $\{u_n\}$ 
in $\W$. 

Then, $\{u_n\}$ is a bounded Palais--Smale sequence for
the functional $I_{\alpha,\beta}$ defined by \eqref{def:I}. 
Indeed, 
$I'_{\alpha_n,\beta_n}(u_n)=0$ and so 
\begin{align*}
\|I_{\alpha,\beta}'(u_n)\|_{W_0^{1,p}(\Omega)^*} &=
\|I_{\alpha,\beta}'(u_n)
-I_{\alpha_n,\beta_n}'(u_n)\|_{W_0^{1,p}(\Omega)^*} 
\\
&\le \frac{|\alpha_n-\alpha|}{p\lambda_1(p)^{1/p}}\|u_n\|_p^{p-1} 
+\frac{|\beta_n-\beta|}{q\lambda_1(p)^{1/q}}\|u_n\|_q^{q-1}
|\Omega|^{1/q-1/p}. 
\end{align*}
On the other hand, by a standard argument based on the $(S_+)$
property of $-\Delta_p$ (cf., \protect{\cite[Lemma
8]{T-2014}}), 
it can be readily shown that
$I_{\alpha,\beta}$ satisfies the bounded
Palais--Smale condition. Hence, $\{u_n\}$ has a subsequence
converging to some critical point $u_0$ of $I_{\alpha,\beta}$. 
Thus, if we show that $u_0\not=0$, then $u_0$ is a positive solution 
of $(GEV;\alpha,\beta)$, whence the proof is complete. 

Now, we will prove that $u_0\not=0$ by way of contradiction. 
Assume that $u_n$ strongly converges to $0$ in $\W$. 
Applying Lemma~\ref{lem:bdd-small} with $\varphi=\varphi_q$, 
we see that 
any $u_n$ satisfies the inequality 
$$
\intO |\nabla \varphi_q|^q|\nabla u_n|^{p-q}\,dx
+\intO |\nabla \varphi_q|^q\,dx 
\ge \intO \left(\alpha_n u_n^{p-q}+\beta_n\right)\varphi_q^q\,dx. 
$$
Letting $n\to\infty$, we have 
$\|\nabla \varphi_q\|_q^q\ge \beta\|\varphi_q\|_q^q$. 
However, this is a contradiction, since
$\lambda_1(q)\|\varphi_q\|_q^q
=\|\nabla \varphi_q\|_q^q\ge \beta\|\varphi_q\|_q^q$ 
and $\beta>\lambda_1(q)$. 

\smallskip
\par\noindent
\textbf{Part (ii).}
From Part (iii) of Proposition \ref{prop:prop} it follows that 
$(GEV;\lambda^*(s)+s, \lambda^*(s)) \equiv (GEV;\lambda_1(q)+s, \lambda_1(q))$ for all $s \geq s^*_+$.
Suppose, contrary to our claim, that $(GEV;\lambda_1(q)+s, \lambda_1(q))$ possesses a positive solution $u$ for some $s > s^*_+$.
As in the proof of Part (iii), Proposition \ref{prop:prop}, we replace $\varphi_p$ by $\varphi_q$ in Picone's identities \eqref{eq:equiv:pic1} and \eqref{eq:equiv:pic2}, and get
\begin{align*}
\notag
&\int_\Omega |\nabla u|^{p-2}  \nabla u \nabla \left( \frac{\varphi_q^p}{u^{p-1}} \right)\, dx
+
\int_\Omega |\nabla u|^{q-2}  \nabla u \nabla \left( \frac{\varphi_q^p}{u^{p-1}} \right)\, dx
\\ 
\notag
&=
(\lambda_1(q) + s) \int_\Omega \varphi_q^p \, dx 
+
\lambda_1(q) \int_\Omega \varphi_q^p u^{q-p} \, dx 
\leq 
\int_\Omega |\nabla \varphi_q|^p \, dx + \lambda_1(q) \int_\Omega \varphi_q^{p} u^{q-p} \, dx,
\end{align*}
which implies that
$$
s \leq \frac{\int_\Omega |\nabla \varphi_q|^p \, dx}{\int_\Omega \varphi_q^p \, dx} - \lambda_1(q) = s^*_+.
$$
However, it is a contradiction, since $s > s^*_+$.

\end{proof*}


\section{Minimax formula for $\lambda^*(s)$}
In this section we prove that definitions \eqref{def:lambda_1} and \eqref{def:lambda_2} are, in fact, equivalent.

\begin{proof*}{Proposition~\ref{prop:equiv}}
Fix any $s \in \mathbb{R}$.
Since $\lambda^*(s)$ is bounded from below by Part (ii) of Proposition \ref{prop:prop}, the definition \eqref{def:lambda_1} implies the existence of a sequence of solutions $\{ u_n \}_{n=1}^{\infty} \in {\rm int}\,\C_+$ (see Remark \ref{rem:positive}) for $(GEV; \lambda_n + s, \lambda_n)$ such that $\lambda_n \to \lambda^*(s)$ as $n \to \infty$ and each $\lambda_n \leq \lambda^*(s)$ (note that here we allow $\lambda_n = \lambda^*(s)$ for all $n \in \mathbb{N}$).

Using $u_n$ as an admissible function for \eqref{def:lambda_2} and noting that 
for any $0\not=\varphi\in C_0^1(\overline{\Omega})_+$ 
the denominator of $\mathcal{L}_s(u_n;\varphi)$ is positive, 
namely, 
$$
\int_\Omega (u_n^{p-1}+u_n^{q-1})\varphi\,dx>0, 
$$
we get
$$
\Lambda^*(s) \geq \inf_{\varphi \in C^1_0(\overline{\Omega})_+ \setminus \{0\}}  \mathcal{L}_s(u_n; \varphi) = \lambda_n \to \lambda^*(s),
$$
and therefore $\Lambda^*(s) \geq \lambda^*(s)$ for any $s \in \mathbb{R}$.

Assume now that there exists $s_0 \in \mathbb{R}$ such that $\Lambda^*(s_0) > \lambda^*(s_0)$. Then, by the definition of $\Lambda^*(s)$, there exist $w \in {\rm int}\,\C_+$ for which we have
$$
\Lambda^*(s_0) \geq \mu := \inf_{\varphi \in C^1_0(\overline{\Omega})_+ \setminus\{0\}} \mathcal{L}_{s_0}(w; \varphi) > \lambda^*(s_0).
$$
However, this implies that $w$ is a positive super-solution of $(GEV; \mu + s_0, \mu)$. Indeed
$$
\mathcal{L}_{s_0}(w; \varphi) \geq \mu > \lambda^*(s_0) \quad \mbox{for all } \varphi \in C^1_0(\overline{\Omega})_+ \setminus\{0\},
$$
and therefore
$$
\intO \left(|\nabla w|^{p-2}+|\nabla w|^{q-2}\right)\nabla w \nabla \varphi\,dx 
-\intO \left((\mu + s_0)|w|^{p-2} + \mu |w|^{q-2} \right) w \varphi\,dx \geq 0 
$$
for any $\varphi\in \W$ satisfying $\varphi(x) \ge 0$ a.e. $x\in\Omega$, due to the approximation arguments. Hence, recalling that $\mu>\lambda^*(s_0)\ge \lambda_1(q)$, Lemma \ref{lem:ss-method} guarantees the existence of a positive solution for $(GEV; \mu + s_0, \mu)$, however it contradicts the definition of $\lambda^*(s_0)$. 
\end{proof*}

\appendix 
\section{The Picone identity for the $(p,q)$-Laplacian}
In this section we prove the variant of Picone's-type identity (cf. \cite{Alleg}), which turns to be useful for problems with $(p,q)$-Laplacian.

First we prepare one auxiliary result.  Denote
\begin{equation}
\label{eq:g}
g(p,q;t):=\frac{(p-1) t^{p-2} + (q-1) t^{q-2}}{(p-1) 
\left(t^{p-1} + t^{q-1}\right)^{\frac{p-2}{p-1}}} 
\end{equation}
\begin{lemma}\label{lem-1} 
Let $1<q,p<\infty$. Then $\inf_{t>0} g(p,q; t) > 0$.
\end{lemma}
\begin{proof} 
Let us denote $\widetilde{m} :=	\min\{p-1,q-1\}$. By standard calculations we get
\begin{align*}
g(p,q;t) &= \frac{(p-1) t^{p-2} + (q-1) t^{q-2}}{(p-1) 
\left(t^{p-1} + t^{q-1}\right)^{\frac{p-2}{p-1}}} = 
 \frac{\widetilde{m} t^{p-2}\left(\frac{p-1}{\widetilde{m}} + \frac{q-1}{\widetilde{m}} t^{q-p}\right)}{(p-1) t^{p-2} \left(1 + t^{q-p}\right)^{\frac{p-2}{p-1}}} \\
 &\geq \frac{\widetilde{m} \left(1 + t^{q-p}\right)}{(p-1) \left(1 + t^{q-p}\right)^{\frac{p-2}{p-1}}} = \frac{\widetilde{m}}{p-1}\left(1 + t^{q-p}\right)^{\frac{1}{p-1}} > \frac{\widetilde{m}}{p-1} > 0
\end{align*}
for all $t > 0$, which completes the proof.
\end{proof} 

\begin{proposition}\label{picone} 
Let $1<q<p<\infty$. Then there exists $\rho > 0$ such that for any differentiable functions $u > 0$ and $\varphi \geq 0$ in $\Omega$ it holds
\begin{equation}
\label{eq:picone}
\left(|\nabla u|^{p-2}  + |\nabla u|^{q-2} \right) \nabla u \nabla \left( \frac{\varphi^p}{u^{p-1} + u^{q-1}} \right) 
\leq \frac{|\nabla \varphi|^p + |\nabla \left(\varphi^{p/q}\right)|^q}{\rho}.
\end{equation}
\end{proposition}
\begin{proof}
First, by standard calculations we get
\begin{align}
\notag
& |\nabla u|^{p-2} \nabla u \nabla \left( \frac{\varphi^p}{u^{p-1} + u^{q-1}} \right)   \\
  \notag
& = p |\nabla u|^{p-2} \nabla u \nabla \varphi \frac{\varphi^{p-1}}{u^{p-1} + u^{q-1}} - 
|\nabla u|^{p} \varphi^p \frac{(p-1) u^{p-2} + (q-1) u^{q-2}}{\left(u^{p-1} + u^{q-1}\right)^2} \\
\label{eq:picone:1}
& \leq
p |\nabla u|^{p-1} |\nabla \varphi| \frac{\varphi^{p-1}}{u^{p-1} + u^{q-1}} - 
|\nabla u|^{p} \varphi^p \frac{(p-1) u^{p-2} + (q-1) u^{q-2}}{\left(u^{p-1} + u^{q-1}\right)^2}
\end{align}
in $\Omega$. 
Applying to the first term Young's inequality 
$$
ab =\frac{a}{\rho^{\frac{p-1}{p}}}\,\rho^{\frac{p-1}{p}} b 
\leq \frac{|a|^p}{p\rho^{p-1}} 
+ \frac{\rho (p-1) |b|^{\frac{p}{p-1}}}{p}
$$
with 
$a= |\nabla \varphi|$,  
$b= |\nabla u|^{p-1} \frac{\varphi^{p-1}}{u^{p-1} + u^{q-1}}$ and any $\rho>0$, 
we obtain 
\begin{align*}
\eqref{eq:picone:1} & \leq 
\frac{|\nabla \varphi|^p}{\rho^{p-1}} + \frac{\rho(p-1) |\nabla u|^{p} \varphi^p}{\left( u^{p-1} + u^{q-1} \right)^{\frac{p}{p-1}}} - 
|\nabla u|^{p} \varphi^p \frac{(p-1) u^{p-2} + (q-1) u^{q-2}}{\left(u^{p-1} + u^{q-1}\right)^2} \\
& = \frac{|\nabla \varphi|^p}{\rho^{p-1}} + 
\frac{\rho(p-1) |\nabla u|^{p} \varphi^p}{\left( u^{p-1} + u^{q-1} \right)^{\frac{p}{p-1}}} \left[ 1 - \frac{(p-1) u^{p-2} + (q-1) u^{q-2}}{\rho(p-1) \left(u^{p-1} + u^{q-1}\right)^{\frac{p-2}{p-1}}} \right] \\
&= \frac{|\nabla \varphi|^p}{\rho^{p-1}} + 
\frac{\rho(p-1) |\nabla u|^{p} \varphi^p}{\left( u^{p-1} + u^{q-1} \right)^{\frac{p}{p-1}}} \left[ 1 - \frac{g(p,q;u)}{\rho}\right]
\end{align*} 
in $\Omega$, 
where $g(p,q;t)$ is defined by \eqref{eq:g}. Since Lemma~\ref{lem-1} 
implies that $\inf_{t>0} g(p,q;t)$ is positive, we can choose $\rho_1 > 0$ small enough to satisfy
$\inf_{t>0} g(p,q;t) \geq \rho_1$. 
This yields $[1-g(p,q;u)/\rho_1] \le 0$ in $\Omega$, and therefore 
\begin{equation}
\label{eq:picone:result:1}
|\nabla u|^{p-2} \nabla u \nabla \left( \frac{\varphi^p}{u^{p-1} + u^{q-1}} \right) 
\leq \frac{|\nabla \varphi|^p}{\rho_1^{p-1}} 
\quad {\rm in}\ \Omega. 
\end{equation}
Similarly, if we choose $\rho_2 > 0$ satisfying $\inf_{t>0} g(q,p;t) \geq \rho_2$, 
then for $\psi=\varphi^{p/q}$ (note $p/q>1$ and $\varphi^p=\psi^q$) we obtain 
\begin{align}
\notag
& |\nabla u|^{q-2} \nabla u \nabla \left( \frac{\psi^q}{u^{p-1} + u^{q-1}} \right) \\
&\le \frac{|\nabla \psi|^q}{\rho_2^{q-1}} + 
\frac{\rho_2 (q-1) |\nabla u|^{q} \psi^q}{\left( u^{p-1} + u^{q-1} \right)^{\frac{q}{q-1}}} \left[ 1 - \frac{g(q,p;u)}{\rho_2}\right] 
\le \frac{|\nabla \psi|^q}{\rho_2^{q-1}}. 
\label{eq:picone:result:2}
\end{align}
Combining now \eqref{eq:picone:result:1} with \eqref{eq:picone:result:2} and taking $\rho := \min\{ \rho_1^{p-1}, \rho_2^{q-1} \}$ we establish the formula \eqref{eq:picone}.
\end{proof}

\section{Non-monotonicity of $\lambda_1(p)$ with respect to $p$}

The main aim of this section is to provide sufficient conditions for  $\lambda_1(p)$ 
to be a non-monotone function w.r.t. $p$.

Throughout this section, we write $\lambda_1(p, \Omega)$ to reflect the dependence of the first eigenvalue $\lambda_1(p)$ on the domain $\Omega$, on which it is defined.

By $B_R$ we denote an open ball in $\mathbb{R}^N$ of radius $R$. We don't fix the center of $B_R$
and write $B_R \subset \Omega$ ($\Omega \subset B_R$) if such center exists.
\begin{proposition} 
Assume that  $r, R \in (1, e)$ and a domain $\Omega \subset \mathbb{R}^N$ ($N \geq 2$) satisfy 
$$
\max\{1,e\ln R\}<r\le R<e \quad {\rm and} \quad 
B_{r} \subset \Omega \subset B_R. 
$$
Then, 
the function $\lambda_1(p, \Omega)$ w.r.t. $p$ has a maximum point $p^*>1$, and hence 
it is non-monotone.

In particular, if $\Omega=B_R$ with $R \in (1,e)$, then $\lambda_1(p,\Omega)$ is non-monotone w.r.t. $p$. 
\end{proposition} 
\begin{proof}
Note that $\lambda_1(p, \Omega)$ is a continuous function w.r.t. $p$ (see \cite[Theorem 2.1]{huang1997}). Hence, to prove that $\lambda_1(p, \Omega)$ possesses a maximum point $p^* > 1$ it is sufficient to show the existence of $p_0 > 1$ such that 
\begin{equation}
\label{eq:lnonm1}
\lambda_1(p_0, \Omega) > \max\left\{\lim_{p \to 1+0} \lambda_1(p, \Omega), \lim_{p \to \infty} \lambda_1(p, \Omega)\right\}.
\end{equation}

First we find the corresponding limits. 
On the one hand, \cite[Theorem 3.1]{huang1997} and \cite[Corollary 5]{benediktdrabek2012} yield that if there exists $r > 1$ such that $B_r \subset \Omega$, then 
\begin{equation}
\label{eq:asympt:above}
\lim_{p \to \infty} \lambda_1(p, \Omega) =  0.
\end{equation}
On the other hand, it is proved in \cite[Corollary 6]{kawohlfridman2003} that
\begin{equation}
\label{eq:asympt:below}
\lim_{p \to 1+0} \lambda_1(p, \Omega) =  h(\Omega),
\end{equation}
where $h(\Omega)$ is the so-called \textit{Cheeger constant} defined by
$$
h(\Omega) := \inf_{D \subset \Omega} \frac{|\partial D|}{|D|}.
$$
Here $|\partial D|$ and $|D|$ are $(N-1)$- and $N$-dimensional Lebesgue
measure of $\partial D$ and $D$, respectively.
Note that Cheeger's constant is known explicitly for some domains; for instance, $B_r$ has $h(B_r) = \frac{N}{r}$ (see \cite{kawohlfridman2003}).

Now to get \eqref{eq:lnonm1} we use the following estimation from \cite[Theorem 2]{benediktdrabek2012}, which holds for any $\Omega \subset B_R$:
\begin{equation}
\label{eq:lp:h}
\lambda_1(p, \Omega) \geq \frac{N p}{R^p}.
\end{equation}
Simple analysis of the function $y(p) = \frac{N p}{R^p}$ shows that if $R \in (1, e)$ then there exists a unique maximum point $p_0 = \frac{1}{\ln R} > 1$ of $y(p)$, and $y(p_0) = \frac{N}{e \ln R}$.

Let us show now the existence of $r, R \in (1, e)$ such that for any $\Omega$ with $B_r \subset \Omega \subset B_R$ it holds $y(p_0) > h(\Omega)$. Then \eqref{eq:asympt:above}, \eqref{eq:asympt:below} and \eqref{eq:lp:h} will imply \eqref{eq:lnonm1}, which proves the assertion of the proposition.

From the monotonicity of Cheeger's constant with respect to a domain (see \cite[Remark 11]{kawohlfridman2003}) it follows that $h(B_r) \geq h(\Omega)$, whenever $B_r \subset \Omega$. Therefore, it is enough to show that
\begin{equation}
\label{eq:lnonm2}
y(p_0) = \frac{N}{e \ln R} > \frac{N}{r} = h(B_r)
\end{equation}
holds for some $r, R \in (1, e)$ with $r<R$. Inequality \eqref{eq:lnonm2} is read as $r > e \ln R$. It is not hard to see that for any fixed $R \in (1, e)$ 
we have $\max\{1,e\ln R\}<R$, since the function 
$\ln t/t$ ($t>0$) has the maximum value $1/e$ only at $t=e$. 

Thus, for any $r, R \in (1, e)$ and $\Omega \in \mathbb{R}^N$ such that
$$
\max\{1,e\ln R\}<r\le R<e \quad {\rm and} \quad 
B_{r} \subset \Omega \subset B_R,
$$
the inequality \eqref{eq:lnonm1} is satisfied for some $p_0 > 1$, and this completes the proof.
\end{proof}

\section{Violation of the assumption {\bf (LI)}}
In this section we give a short one-dimensional example indicating that, in general, the first eigenvalues $\lambda_1(p, m_p)$ and $\lambda_1(q, m_q)$ of zero Dirichlet $-\Delta_p$ and $-\Delta_q$ on $\Omega$ with weights $m_p$ and $m_q$, respectively, can have the same eigenspace, that is, $\varphi_1(p, m_p) \equiv k \varphi_1(q, m_q)$ in $\Omega$ for some $k \neq 0$.

Let $u$ be a positive $C^2$-solution of 
$$
\left\{
\begin{aligned}
-u'' &= |u|^{p-2}\, u \quad {\rm in}\ (0, \pi), \\
u(0) &= u(\pi) = 0,
\end{aligned}
\right.
$$
where $p > 2$. Existence and regularity of such a solution is a classical example in various textbooks on nonlinear analysis (see, e.g., \cite[Example~7.4.7, p.~485]{drabekmilota}).

It is easy to see, that $u$ is also the first eigenfunction of zero Dirichlet $-\Delta$ on $(0, \pi)$ with the weight function $m_2(x) = |u(x)|^{p-2}$, with the corresponding eigenvalue $\lambda_1(2, m_2) =1$. Note that $m_2 > 0$ in $(0, \pi)$ and $m_2 \in L^{\infty}[0, \pi]$, since $p > 2$.

At the same time, $u$ becomes the first eigenfunction of $-\Delta_p$ on $(0, \pi)$ with weight $m_p(x) = (p-1) |u'(x)|^{p-2}$, with the eigenvalue $\lambda_1(p, m_p) =1$. Indeed, 
$$
\left\{
\begin{aligned}
-(|u'|^{p-2} u')' &\equiv -(p-1)|u'|^{p-2} u'' \\
				  &= (p-1) |u'|^{p-2} |u|^{p-2} u = m_p(x) |u|^{p-2} u \quad {\rm in}\ (0, \pi), \\
u(0) = u(\pi) &= 0.
\end{aligned}
\right.
$$
Moreover, $m_p \in L^{\infty}[0, \pi]$, since $p > 2$ and $u \in C^2[0, \pi]$, and $m_p \geq 0$ in $[0, \pi]$.

Therefore, $\lambda_1(2, m_2)$ and $\lambda_1(p, m_p)$ have the same eigenspace, i.e. {\bf (LI)} is violated.



\end{document}